\documentclass[10pt]{amsart}
\usepackage[english]{babel}
\usepackage[T1]{fontenc}
\usepackage[all]{xy}
\usepackage{etex}
\usepackage{pifont}
\usepackage{rotating}
\usepackage{pdfpages}
\usepackage{pdfsync}
\usepackage{hyperref}
\hypersetup{
    colorlinks,%
    citecolor=black,%
    filecolor=black ,%
    linkcolor=black,%
    urlcolor=black
}

\usepackage{appendix}
\usepackage{graphicx}
\usepackage{amssymb,amsmath,amsthm,amscd,mathrsfs,float,mathdots}
\usepackage{geometry}
\addtocounter{section}{0}             % Start with section 1
\numberwithin{equation}{section}       % Number formulas within sections

\addtocounter{section}{0}             % Start with section 1
\numberwithin{equation}{section}
\theoremstyle{plain} % style plain

\numberwithin{equation}{section}
\newtheorem{theorem}[equation]{Theorem}
\newtheorem{definition}[equation]{Definition}
\newtheorem{lemma}[equation]{Lemma}
\newtheorem{corollary}[equation]{Corollary}
\newtheorem{proposition}[equation]{Proposition}

\newtheorem{remark}[equation]{Remark}

\newcommand{\Z}{\mathbb{Z}}
\newcommand{\Q}{\mathbb{Q}}

\newcommand{\K}{\mathbf{k}}
\newcommand{\F}{\mathbb{F}}
\newcommand{\GL}{\mathbf{GL}}
\newcommand{\G}{\mathbb{G}}
\newcommand{\mcI}{\mathcal{I}}

\newcommand{\Filt}{\mathrm{Filt}}
\newcommand{\Vect}{\mathrm{Vect}}
\newcommand{\IC}{\mathrm{IC}}

\newcommand{\m}{\mathfrak{m}}

\newcommand{\locring}{\mathcal{O}}

\newcommand{\qelbar}{\overline{\Q}_{\ell}}

\newcommand{\weyl}[1]{\ensuremath{\widetilde{W}_{#1}}}

\newcommand{\iwahorihecke}[1]{\ensuremath{\mathcal{H}_{I_{#1}}}}
\newcommand{\flagvar}[1]{\ensuremath{\mathcal{F}l_{#1}}}

\newcommand{\heckefunc}[1]{\ensuremath{\overset{\leftarrow}{H}_{#1}}}
\newcommand{\heckefuncr}[1]{\ensuremath{\overset{\rightarrow}{H}_{#1}}}
\newcommand{\I}[1]{\ensuremath{I_{#1}}}

\newcommand{\Gstrat}{{}_{{}_{s_{1},s_{2}}}G(F)}
\newcommand{\Flagstrat}{{}_{{}_{s_{1},s_{2}}}\mathcal{F}l_{G}}

\newcommand{\newtimes}{\tilde{\times}}
\newcommand{\bt}{\tilde{\boxtimes}}
\newcommand{\iso}{\widetilde{\longrightarrow}}

\newcommand{\Pstrat}{{}_{{}_{s_{1},s_{2}}}P(F)}
\newcommand{\FlagstratP}{{}_{{}_{s_{1},s_{2}}}\mathcal{F}l_{P}}
\newcommand{\Lstrat}{{}_{{}_{s_{1},s_{2}}}L(F)}
\newcommand{\FlagstratL}{{}_{{}_{s_{1},s_{2}}}\mathcal{F}l_{L}}
\newcommand{\isom}{{\widetilde\to}}
\newcommand{\on}{\operatorname}

\newcommand{\Hom}{\on{Hom}}
\newcommand{\dimrel}{\on{dim.rel}}

\newcommand{\Fl}{{\mathcal{F}l}}
\newcommand{\wt}{\widetilde}
\newcommand{\Qlb}{\mathbb{\bar Q}_\ell}

\newcommand{\Gr}{\on{Gr}}
\newcommand{\hook}[1]{\stackrel{#1}{\hookrightarrow}}
\newcommand{\SL}{\on{SL}}

\newcommand{\cI}{\mathcal{I}}
\newcommand{\cO}{\mathcal{O}}

\newcommand{\cH}{\mathcal{H}}
\newcommand{\cS}{\mathcal{S}}

\newcommand{\ZZ}{\mathbb{Z}}

\begin{document}
\selectlanguage{english}
\title[ Geometric local theta correspondence for $(\GL_n,\GL_m)$]{Geometric local theta correspondence for dual reductive pairs of type II at the Iwahori level}
\author{Banafsheh Farang-Hariri}

\address{Humboldt-Universit\"et zu Berlin, Mathematisch-Naturwissenschaftliche Fakult\"at II \, \, \, \, \, \, \, \, \, \, \, \, \, \,   Institut f\"ur Mathematik, Sitz: Rudower Chaussee 25 (Adlershof), D-10099 Berlin}
\email{bfhariri@gmail.com}
\parindent=0cm
\begin{abstract}
In this paper we are interested in the geometric local theta correspondence at the Iwahori level for dual reductive pairs $(G,H)$ of type II over a non-Archimedean field of characteristic $p\neq 2$ in the framework of the geometric Langlands program.  We consider the geometric version of the $I_{H}\times I_{G}$-invariants of the Weil representation $\mathcal{S}^{I_{H}\times I_{G}}$ as a bimodule under the of action Iwahori-Hecke algebras $\iwahorihecke{G}$ and $\iwahorihecke{H}$ and we give some partial geometric description of the corresponding category under the action of Hecke functors. We also define geometric Jacquet functors for any connected reductive group $G$ at the Iwahori level and we show that they commute with the Hecke action of the $\iwahorihecke{L}$-subelgebra of $\iwahorihecke{G}$ for a Levi subgroup $L$.

\end{abstract}

\maketitle
{\small\tableofcontents}
\section{Introduction} 
Let $\K$ be a finite field $\F_{q}$ of characteristic different from $2$, let $F=\K((t))$ and $\cO=k[[t]]$. All representations will be assumed to be smooth and will be defined over $\qelbar$, where $\ell$ is a prime number different from the characteristic of $F$. The basic notions of the Howe correspondence from the classical point of view have been presented in \cite{MVW}, see also \cite{Kudla}. Let $(G,H)$ be a split dual reductive pair in some symplectic group $Sp(W)$ over $\K$ and let  $\widetilde{Sp}(W)$ be the metaplectic group which is the twofold topological covering of the symplectic group $Sp(W)$. Let $\cS$ be the Weil representation of $\widetilde{Sp}(W)$. Assume that the metaplectic cover $\widetilde{Sp}(W)\to Sp(W)$ admits a section over $G(F)$ and $H(F)$. Then the local theta correspondence (known also as Howe correspondence) is a correspondence between some class of  representations of $G(F)$ and some class of representations of $H(F)$. It is well-known that Howe correspondence  realizes Arthur-Langlands functoriality in some special cases. From the classical point of view, we refer to \cite{Adams}, \cite{Arthur}, \cite{Kudla} \cite{Rallis} and from the geometric point of view, we refer to  \cite{Lysenko1}, \cite{BFH2}.  It is of great interest  to understand the geometry underlying the Howe correspondence and establish its analogue in the framework of the geometric Langlands program (see \cite{Frenkel-Gaitsgory}, \cite{Bezrukavnikov}, \cite{Lysenko-Lafforgue}, \cite{Lysenko1}, \cite{BFH2}). The unramifed geometric Howe correspondence for dual pairs $(\mathbf{Sp}_{2n}, \mathbf{S}\mathbf{O}_{2m})$ and $(\GL_n,\GL_m)$ has been  studied in \cite{Lysenko1}. One of our motivations is to extend the results of \cite{Lysenko1} to the Iwahori (tamely ramfied) case in the geometric setting and complete the description of the Howe correspondence for dual reductive pairs of type II already initiated in \cite{BFH2}. This will be a new step towards proving the relation between Howe correspondence and Arthur-Langlands functoriality conjecture for dual reductive pairs of type II announced in \cite[Conjecture 1.2]{BFH} already proved for dual pairs $(\GL_1, \GL_m)$ for all $m\geq 1$. 
\par\medskip
In the sequel we will restrict ourselves to the dual reductive pairs of type II. More Precisely, let $L_{0}$ (resp. $U_{0}$) be a $n$-dimensional (resp. $m$-dimensional) $\K$-vector  space with $n\leq m$, and let $G=\GL(L_{0})$ and $H=\GL(U_{0})$. Denote by $\Pi(F)$ the space $(U_{0}\otimes L_{0})(F)$ and $\mathcal{S}(\Pi(F))$ the Schwartz space of locally constant functions with compact support on $\Pi(F)$. This space realizes the restriction of the Weil representation to $G(F)\times H(F)$. According to Howe and \cite{Minguez1}, we know that the Howe correspondence associates to any smooth irreducible representation $\pi$ of $G(F)$ appearing as a quotient of the restriction of the Weil representation a unique smooth irreducible non-zero representation of $H(F),$ denoted by $\theta_{n,m}(\pi),$ such that $\pi\otimes \theta_{n,m}(\pi)$ is a quotient of the restriction of the Weil representation to $G(F)\times H(F)$.
\par\medskip
One of the most interesting classes of representations to be considered for the study of the Howe correspondence is the class of tamely ramified representations. A representation of $G(F)$ is said to be tamely ramified if it admits a non zero vector fixed  under an Iwahori subgroup $I_{G}$ of $G(F)$. Let us consider the functor sending any tamely ramified representation $V$ of $G(F)$ to its space of invariants $V^{I_{G}}$ under $I_{G}$. Then, the latter  is naturally a module over the Iwahori-Hecke algebra $\iwahorihecke{G}$. According to \cite[Theorem 4.10]{Borel1} this functor is an equivalence of categories between the category of tamely ramified admissible representations of $G$ and the category of finite-dimensional modules over $\iwahorihecke{G}$. Moreover, this functor is exact over the category of smooth representations of $G(F)$. In the tamely ramified case, we can interpret the Howe correspondence in the language of modules over Iwahori-Hecke algebras. The space $\cS(\Pi(F))^{I_H\times I_G}$ of $I_{H}\times I_{G}$-invariants  in the Weil representation $\mathcal{S}(\Pi(F))$ is naturally a bimodule under the action of Iwahori-Hecke algebras $\iwahorihecke{G}$ and $\iwahorihecke{H}.$ We would like to understand this module structure by geometric means. The geometric analogue of the Schwartz space $\cS(\Pi(F))^{I_H\times I_G}$ and the action of Iwahori-Hecke algebras of $G$ and $H$ on it have been already constructed in \cite[\S 3]{BFH2}. Namely, in the geometric setting the space $\cS(\Pi(F))^{I_H\times I_G}$ is the category $P_{I_H\times I_G}(\Pi(F))$ of $I_H\times I_G$-equivariant perverse sheaves on $\Pi(F)$, its precise definition involves some limit procedure. Denote by $D_{I_{H}\times I_{G}}(\Pi(F))$ the derived category of $\ell$-adic $I_{H}\times I_{G}$-equivariant sheaves on $\Pi(F)$ constructed in \cite{BFH2}. The action of Iwahori-Hecke algebras is geometrized to an action of Hecke functors $\heckefunc{G}$ and $\heckefunc{H}$. Denote by $\flagvar{G}$ the affine flag variety of $G$ and by $P_{I_{G}}(\flagvar{G})$ the category of $I_{G}$-equivariant perverse sheaves on $\flagvar{G}$. These Hecke functors define an action of $P_{I_{G}}(\flagvar{G})$ and $P_{I_{H}}(\flagvar{H})$ on  $D_{I_H\times I_G}(\Pi(F))$. This geometrization has actually been  done in a more general setting of any dual reductive pair and at the level of derived categories in \cite{BFH2}. While in \cite[\S 7]{BFH2} we gave an explicit description of the bimodule $P_{I_H\times I_G}(\Pi(F))$ in the case of $n=1$ and $m\geq 1$, in this article, we obtain some partial results towards the description of the category $P_{I_H\times I_G}(\Pi(F))$ as a module over the Hecke functors for any $n\leq m$. 
\par\medskip
One of the key steps in \cite{BFH2} is the description of the simple objects of the category  $P_{I_{H}\times I_{G}}(\Pi(F))$ that we will use in this paper. Let us recall this result. Let $S_{n,m}$ denote the set of pairs: a subset $I_{s}\subset\{1,\ldots, m\}$ of $n$ elements and a bijection $s: I_{s}\to\{1,\ldots, n\}$. For $N,r$ two integers such that $N+r>0$, let $\Pi_{N,r}=t^{-N}\Pi/t^r\Pi$. Fix a maximal torus $T$ in $ G$ and a Borel subgroup $B$ containing $T$. Denote $X_{G}$ the lattice of cocharacters of $T$. For each pair $(\lambda,s)$  in $X_{G}\times S_{n,m},$ we have introduced some subvarieties $\Pi_{N,r}^w$ in $\Pi_{N,r}$ for $N,r$ large enough and we obtained the following result \cite[Theorem 6.6]{BFH2}: 
the  simple objects of  the category $P_{I_{H}\times I_{G}}(\Pi(F))$ are parametrized by $X_{G}\times S_{n,m}.$ For any element $w=(\lambda,s)$ in $X_{G}\times S_{n,m},$ the  irreducible object of $P_{I_{H}\times I_{G}}(\Pi(F))$ indexed by $w$ is the intersection cohomology sheaf $\cI^w$ of $\Pi_{N,r}^{w}$ for $N,r$ large enough. We also introduced the objects $\cI^{w !}$, which are extensions by zero of the constant perverse sheaf under $\Pi_{N,r}^{w}\hook{} \Pi_{N,r}$. Our aim is to understand as possible as we can the action of Hecke functors on these simple objects.
\par\medskip
\subsection*{Summary of results} $ $
We construct a filtration on the category  $P_{I_H\times I_G}(\Pi(F))$ compatible with the Hecke functors which enables us to control this action. We study some submodules of $P_{I_H\times I_G}(\Pi(F))$ and give a precise description of those under the action of Hecke functors.  Particularly we construct the geometric version of the first term of Kudla's filtration and we show that it can be identified  with the induced representation from a parabolic subalgebra of $\iwahorihecke{H}$ by geometric means. Kudla's filtration is a key ingredient in the study of Howe correspondence for dual pairs of type II in the classical setting \cite{Minguez1}. It is also used in the study of Howe correspondence for dual reductive pairs of type I by Kudla \cite{Kudla}.
\par\medskip
We also construct a geometric version of Jacquet functors of the Weil representation at the Iwahori level and show that they commute with the Hecke action of the subalgebra $\iwahorihecke{L}$ of $\iwahorihecke{G}$ for some Levi subgroup $L$. The Jacquet functors of the Weil representation have been studied in the classical representation theory by Kudla \cite{Kudla} and Rallis \cite{Rallis} at the unramified level. A part of these classical results at unramified level have been already geometrized in \cite{Lysenko1}. Our construction is an extension of geometric Jacquet functors at the unramified level obtained in \cite[Corollary 3]{Lysenko1} to the Iwahori level.  In \cite{Lysenko1}, one of the key results used to prove the commutativity of the Hecke actions and Jacquet functors is the hyperbolic localization functor of Braden introduced in \cite{Braden}. There is another construction of geometric Jacquet functors in the case of real reductive groups due to Emerton-Nadler-Vilonen using D-modules and nearby cycles on the flag variety \cite{ENV}. Although the construction in \cite{ENV} is done in an algebraic way, the underlying geometric interpretation is explained in \cite[\S 5]{ENV} and relies on the hyperbolic localization used in \cite{Lysenko1} . THis is also the case in our construction and can  be seen in Corollary $\ref{hyploc}$.	
\par\medskip
Let us briefly discuss how the paper is organized.
In section $\ref{filtration}$,  assuming $n\leq m$, we introduce a filtration on $P_{I_H\times I_G}(\Pi(F))$ indexed by $\ZZ$ and show that it is compatible with the natural grading of $P_{I_G}(\Fl_G)$ given by the connected components of $\Fl_G$ and by the action of Hecke functors. This filtration on $P_{I_H\times I_G}(\Pi(F))$  is expected to be compatible with the filtration already studied on the conjectural bimodule $K(\mathcal{X})$ in \cite[\S $8$]{BFH2} describing the geometric local Arthur-Langlands functoriality at the Iwahori level for some map between corresponding dual Langlands groups.    
\par\medskip
Consider the element $w=(0, (I, w_0))\in X_G\times S_{n,m}$, where $I=\{1,\ldots, n\}$ and $w_0: I\to \{1,\ldots, 0\}$ is the longest element of the finite Weyl group of $G$.  We obtain some results on the submodule in $P_{I_{H}\times I_{G}}(\Pi(F))$ generated by $\cI^{w_{0}}$ (resp., $\cI^{w_{0}!}$). In subsection $\ref{nm}$, we consider the case $n=m$. The submodule in $D_{I_H\times I_G}(\Pi(F))$ generated by $\cI^{w_{0}!}$ over $P_{I_G}(\Fl_G)$ is free of rank one (and is also preserved by $P_{I_H}(\Fl_H)$). We also precise an equivalence of categories $\tilde\sigma: P_{I_H}(\Fl_H)\,\isom\, P_{I_G}(\Fl_G)$ which defines at the level of functions  an anti-involution of the Iwahori-Hecke algebras $\iwahorihecke{G}$ and $\iwahorihecke{H}$. By means of this equivalence we relate the actions of Hecke functors for $H$ and $G$ on the submodule generated by $\cI^{w_{0}}$ (resp. $\cI^{w_{0}!}$). In subsection $\ref{theta}$, we assume $n\leq m$ and we consider the submodule $\Theta$ in the Grothendieck group $K(D_{I_H\times I_G}(\Pi(F)))$ of $D_{I_H\times I_G}(\Pi(F)))$ under the action of  $K(P_{I_G}(\Fl_G))$  generated by the elements $\mathcal{I}^{\mu!},$ where $I_{\mu}$ runs through all possible subsets of $n$ elements in $\{1,\dots ,m\}$.  We show that the submodule $\Theta$ is free of rank $C_{m}^{n}$ over $K(P_{I_G}(\Fl_G))$. The elements $\mathcal{I}^{\mu!}$ form a basis of this module over $K(P_{I_G}(\Fl_G))$. The submodule $\Theta$ is a key object in the proof of the classical Howe correspondence. It is indeed the first term of the Kudla's filtration defined over the Weil representation in \cite{MVW}. The considerations in this subsection are essentially on the level of Grothendieck groups, we formulate them on the level of derived categories however when this is possible. 
Let $\cS_{0}$ be the $\Qlb$-subspace  of  $K(D_{I_H\times I_G}(\Pi(F))\otimes\Qlb$ generated by elements of the form $\cI^{(w\centerdot w_0)!}$, where  $w$ runs through $\widetilde{W}_G$ the affine extended Weyl group of $G$. The space  $\cS_0$  is a free module  of rank one over $\cH_{I_G}$.   We consider the standard Levi subgroup $M$ of $H$ corresponding to the partition $(n,m-n)$ of $m$ and we recall briefly the construction of the subalgebra $\iwahorihecke{M}$  of the Iwahori-Hecke algebra $\iwahorihecke{H}$ and some properties according to \cite{Prasad}. Then we endow $\mathcal{S}_{0}$ with a right action of $\iwahorihecke{M}$ and by parabolic induction we construct an induced module.  We obtain the two following results:  The space $\mathcal{S}_{0}$ is a submodule of $K(D_{I_{G}\times I_{H}}(\Pi(F)))\otimes \qelbar$ for the right action of $\iwahorihecke{M}.$  The adjunction map $\alpha:\mathcal{S}_{0}\otimes_{\iwahorihecke{M}} \iwahorihecke{H} \to \mathcal{S}^{I_{H}\times I_{G}}(\Pi(F))$ is injective and its image equals $\Theta\otimes \qelbar.$ In the rest of this subsection we show that the action of the Iwahori-Hecke algebra of the factor  $\GL_{n}$ of $M$ identifies  with the action of $\iwahorihecke{G}$ via the anti-involution $\tilde{\sigma}$ defined in Theorem $\ref{sig}$. The action of the Iwahori-Hecke algebra  of $\GL_{m-n}$ is  by shifting by $[-\ell(w)]$, where $\ell$ denotes the length function on $\weyl{G}$.  
\par\medskip
At last in section $\ref{jacquet}$, we construct a geometric analogue of  Jacquet functors of the Weil representation in the Iwahori case. These functors are a key step in the geometric proof of the Howe correspondence in the unramified case in \cite{Lysenko1}. Moreover, we prove that they are compatible with the Hecke action of the $\iwahorihecke{L}$-subelgebra of $\iwahorihecke{G}$. We also show that they preserve pure perverse sheaves of weight zero. 
\par\medskip
In Appendix $\ref{Appendix}$ we recall the construction of Hecke functors from \cite{BFH2} and in Appendix $\ref{Appendix2}$ we present a complete calculation of the Hecke functor corresponding to $H$ in the special case of objects $\mathcal{I}^{\mu}$.
\par\medskip
\textbf{Acknowledgement}: The author would like to thank her advisor Sergey Lysenko for initiating her to this subject and sharing his insights and ideas and also acknowledge support from DFG International Research Training Group GRK1800 \textit{Moduli and automorphic forms}.
\section{Notation and setup}
\label{notation}
Let $\K$ be an algebraically closed field of characteristic $p> 2$ and let $F=\K((t))$ be the field of Laurent series with coefficients in $\K$ and $\locring=\K[[t]]$ be its  ring of integers. Denote by $\ell$ a prime number different from $p$. Let $L_{0}$ (resp. $U_{0}$) be a $n$-dimensional (resp. $m$-dimensional) $\K$-vector  space with $n\leq m$, and let $G=\GL(L_{0})$ and $H=\GL(U_{0})$.  
Denote by $\{e_{1},\dots,e_{n}\}$ the standard basis of $L_{0}$ and $\{u_{1},\dots,u_{m}\}$ the standard basis of $U_{0}$ and $\{u^{*}_{1},\dots,u^{*}_{m}\}$ its dual basis. Denote by $\Pi(F)$ the space $(U_{0}\otimes L_{0})(F)$ and $\mathcal{S}(\Pi(F))$ the Schwartz space of
$\qelbar$-valued locally constant functions with compact support on $\Pi(F)$. Let $T_{G}$ (resp. $T_{H}$) be the standard maximal torus of diagonal matrices in $G$ (resp. $H$) and $B_{G}$ (resp. $B_{H}$) be the Borel subgroup of upper triangular matrices containing $T_{G}$ (resp. $T_{H}$).  Denote by $I_{G}$ (resp. $I_{H}$) the Iwahori subgroup of $G(F)$ (resp. $H(F)$) corresponding to the standard Borel subgroup $B_{G}$ (resp. $B_{H}$).  Denote by $(\check{X}_{G},\check{R},X_{G},R, \Delta)$ the root datum associated with $(G,T_{G}, B_{G})$. Throughout this article we denote by $\check{X}_{G}$ the characters of $T_{G}$ and $X_{G}$ denotes the cocharacter lattice of $T_{G}$. The set $\check{R}$ is the set of roots and $R$ is the set of coroots and $\Delta_{G}$ is the basis formed by simple roots. If there is no ambiguity we will omit the subscript ${}_G$.  We denote by $W_{G}$ the finite Weyl group associated with the root datum $(\check{X}_{G},\check{R},X_{G},R)$. Let $\weyl{G}$ be  the affine extended Weyl group which is the semi-direct product $W_{G}\ltimes X_{G}.$  Denote by $\ell$ the length function on $\weyl{G}$. Let $X_{G}^{+}$ be the set of dominant elements in $X_{G}$. 
\par\medskip
For any scheme or stack locally of finite type over $\K$, we denote by $D(S)$  the bounded derived category of constructible $\qelbar$-sheaves over $S$. Write $\mathbb{D}$ for the Verdier duality functor and we denote by $P(S)$ the full subcategory of 
perverse sheaves in $D(S)$.  Denote by $K(P(S))$ the Grothendieck group of the category $P(S)$. Let $X$ be a scheme of finite type over $\K$. For $Z$ a smooth $d$-dimensional irreducible locally closed subscheme of $X$ and $i : Z\to X$ the corresponding immersion, we define the intersection cohomology
sheaf (IC-sheaf for short), $\IC(Z)$ as the perverse sheaf $i_{Z!*}(\qelbar)[d]$.
\par\medskip
 Assume temporary that the ground field $\K$ is the finite field $\F_{q}.$
Denote by $\iwahorihecke{G}$ the Iwahori-Hecke algebra of $G$ which is the space of locally constant, $I_{G}$-bi-invariant compactly supported $\qelbar$-valued functions on $G(F)$ endowed with the convolution product. There are two well-known  presentations of this algebra by generators and relations. The first is due to Iwahori-Matsumoto \cite{Iwahori-mat} and the second is by Bernstein in \cite{Lu1} and \cite{Lu2}. We will use the second one. 
\par\medskip
Denote by $Gr_{G}$ the affine Grassmanian associated with $G$. The $G(\mathcal{O})$-orbits on $Gr_{G}$ are parametrized by $W_{G}$-orbits in $X_{G}$ and for a given $\lambda$ in $X_{G}$, the $G(\locring)$-orbit associated to $W_{G}.\lambda$ is $G(\locring).t^{\lambda}$ denoted by $\Gr_{G}^{\lambda}$, where $t^{\lambda}$ is the image of $t$ under the map $\lambda:F^{*}\to G(F)$.  The $I_G$-orbits on $Gr_{G}$ are parametrized by cocharacters $\lambda$ in $X_{G}$. For any $\lambda$ in $X_{G}$, denote by $O^{\lambda}$ the $I_{G}$-orbit through $t^{\lambda}G(\locring)$ in $Gr_{G}$ and by $\overline{O^{\lambda}}$ its closure. Each orbit is an affine space. The category $P_{G(\locring)}(Gr_{G})$ of $G(\mathcal{O})$-equivariant perverse sheaves on $Gr_{G}$ is endowed with a geometric convolution product making it a symmetric monoidal category \cite{Mirkovic}. 
Denote by $\flagvar{G}$ the affine flag variety associated with $G$ and denote by  $P_{I_{G}}(\flagvar{G})$ the category of $I_{G}$-equivariant  perverse sheaves on $\flagvar{G}$. The category  $D_{I_{G}}(\flagvar{G})$ is endowed with the geometric convolution denoted by $\star$, (see \cite{Gaitsgory}), and we have $K(P_{I_{G}}(\flagvar{G}))\otimes \qelbar\iso \iwahorihecke{G}$. 
\par\medskip
 Let $R$ be a $\K$-algebra. A complete periodic flag of lattices inside $R((t))^{n}$ is a flag
$$L_{-1} \subset L_{0} \subset L_{1} \subset \dots$$
 such that each $L_{i}$ is a lattice in $R((t))^{n},$   each quotient $L_{i+1}/L_{i}$ is a locally free $R$-module of rank one and $L_{n+k}=t^{-1}L_{k}$ for any $k$ in $\Z.$ For $1\leq i\leq n$, set 
 $$\Lambda_{i,R}=(\oplus_{j=1}^{i} t^{-1}R[[t]]e_{j})\oplus(\oplus_{j=i+1}^{n}R[[t]]e_{j}).$$ 
For all $i$ in $\Z$, we set $\Lambda_{i+n,R}=t^{-1}\Lambda_{i,R}$. 
This defines the standard complete lattice flag 
$$\Lambda_{-1,R}\subset\Lambda_{0,R}\subset\Lambda_{1,R}\subset \dots $$
denoted by $\Lambda_{\bullet,R}$ in $R((t))^{n}.$  
For any $\K$-algebra $R$, the set $\flagvar{G}(R)$ is naturally in bijection with the set of complete periodic lattice flags in $R((t))^{n}$ and is an ind-scheme.  
%The affine flag variety decomposes as a disjoint union
%\[\flagvar{G}=\bigcup_{w\in{\weyl{G}}}I_{G}wI_{G}/I_{G}.\]
%The closure of each Schubert cell, $I_{G}wI_{G}/I_{G}$ is a union of Schubert cells and the closure relations are given by the Bruhat order:
%\[\overline{I_{G}wI_{G}/I_{G}}=\bigcup_{w^{'}\leq w}I_{G}w^{'}I_{G}/I_{G}.\]
\par\medskip
Assume that $\K$ is finite.
For any $w\in{\weyl{G}}$  we will denote the Schubert cell $I_{G}wI_{G}/I_{G}$ by $\flagvar{G}^{w}$. It is  isomorphic to $\mathbb{A}^{\ell(w)}.$  For $w\in \weyl{G},$ denote by $j_{w}$ the inclusion of $\flagvar{G}^{w}$ in $\flagvar{G},$ and let  $L_{w}=j_{w!*}\qelbar[\ell(w)](\ell(w)/2)$ be the IC-sheaf of $\flagvar{G}^{w}.$  We write
 $L_{w!}=j_{w!}\qelbar[\ell(w)](\ell(w)/2)$ and $L_{w*}=j_{w*}\qelbar[\ell(w)](\ell(w)/2)$ for the standard and costandard objects.  They satisfy $\mathbb{D}(L_{w*})=L_{w!}$. Remark that in the notation of $L_{w!}$ and $L_{w*}$  we wrote the Tate twists as we assumed that we are working on a finite field. To any element $\mathcal{G}$ in $P_{I_G}(\flagvar{G})$ we attach a function  $[\mathcal{G}]:G(F)/I_{G}\longrightarrow \qelbar$  given by $[\mathcal{G}](x)=Tr(Fr_{x},\mathcal{G}_x),$ for $x$ a point in $G(F)/I_{G}$ and $Fr_x$ is the geometric Frobenius at $x$.  The function $[\mathcal{G}]$ is an element of the $\iwahorihecke{G}.$ In particular $[L_{w!}]=(-1)^{\ell(w)}q_{w}^{-1/2}T_w$ and $[L_{w*}]=(-1)^{\ell(w)}q_{w}^{1/2}T_{w^{-1}}^{-1},$ where $q_{w}=q^{\ell(w)}.$ Here $T_{w}$ denotes the characteristic function of the double coset $I_{G}wI_{G}.$ Remark that in this paper as we will work over an algebraically closed field, we will forget the Tate twists. 
The map  sending  $\lambda$  to $L_{t^{\lambda}*},$  for any $\lambda$ in $X^{+}_{G}$ extends naturally to a monoidal functor $$\mathrm{R}(T)\longrightarrow{D_{I_{G}}(\flagvar{G})}.$$
The image of $\lambda$ under the above functor  is usually called a Wakimoto sheaf.  There are two conventions for defining the Wakimoto sheaves. The first convention is due to Bezrukavnikov in \cite{Arkhipov}. We will use the convention due to Prasad in \cite{Prasad} by letting $\Theta_{\lambda}=L_{t^{\lambda !}}$ for $\lambda$ dominant and $\Theta_{\lambda}=L_{t^{\lambda *}}$ for $\lambda$ anti-dominant.  In any case Wakimoto sheaves verify the following: $\lambda\in{X}$, if $\lambda=\lambda_{1}-\lambda_{2}$ where $\lambda_{i}$ are dominant for $i=1,2,$ then $\Theta_{\lambda}\simeq\Theta_{\lambda_{1}}\star\Theta_{-\lambda_{2}}.$
 According to \cite[Theorem 5]{Arkhipov}, these are actually objects of the category $P_{I_{G}}(\flagvar{G})$ (a priori they are defined as objects of the triangulated category $D_{I_{G}}(\flagvar{G}))$).  
\par\medskip
As mentioned above the space  $\mathcal{S}(\Pi(F))^{I_{H}\times I_{G}}$ is naturally a module over Iwahori-Hecke algebras $\iwahorihecke{G}$  and $\iwahorihecke{H}$ of $G$ and $H$. The action is defined  by convolution.  The geometric analogue of the $(\iwahorihecke{G},\iwahorihecke{H})$-bimodule $\mathcal{S}(\Pi(F))^{I_{H}\times I_{G}}$ is constructed in \cite[\S 3]{BFH2} that is  the category $P_{I_{H}\times I_{G}}(\Pi(F))$ of $I_{H}\times I_{G}$-equivariant perverse sheaves on $\Pi(F)$ in the derived category $D_{I_{H}\times I_{G}}(\Pi(F))$ under the action of the two Hecke functors:
$$\heckefunc{G}:P_{I_{G}}(\flagvar{G})\times D_{I_{H}\times I_{G}}(\Pi(F))\longrightarrow D_{I_{H}\times I_{G}}(\Pi(F))$$
and 
$$\heckefunc{H}:P_{I_{H}}(\flagvar{H})\times D_{I_{H}\times I_{G}}(\Pi(F))\longrightarrow D_{I_{H}\times I_{G}}(\Pi(F)).$$
For the sake of the reader, we recall briefly the construction of these Hecke functors in Appendix $\ref{Appendix}$ following \cite{BFH2}.  The goal is to understand these two Hecke functors as much as possible.
\par\medskip
For any two integers $N,r$ such that $N+r>0$, let $\Pi_{N,r}=t^{-N}\Pi/t^r\Pi$. 
Let $U^*$ be the dual of $U$.   A point $v$ in  $\Pi(F)$ may be seen as a $\locring$-linear map $v:U^{*}\to L(F).$ For $v$ in $\Pi_{N,r},$ let $U_{v,r}=v(U^*)+t^{r}L.$ Then $U_{v,r}$  is a $\locring$-module in $L(F)$ and may be seen as a point of $Gr_{G}$.  Let $\check{\omega}_{1}=(1,0\dots,0)$ be the highest weight of the standard representation of $G$, let $w_0$ be the longest element of the finite Weyl group $W_{G}$. For $\lambda\in X_{G}$ such that for any $\nu$ in $W_{G}.\lambda$
\begin{equation}
\label{condition1}
\langle\nu,\check{\omega}_{1}\rangle\leq{r}\hspace{1cm} \textrm{and}\hspace{1cm}\langle -\nu,\check{\omega}_{1}\rangle\leq N,
\end{equation}
let $\Pi_{\lambda,r}\subset \Pi_{N,r}$ be the locally closed subscheme of $v\in \Pi_{N,r}$ such that $U_{v,r}$ lies in $I_G t^ {\lambda} G(\locring)$. According to \cite{BFH2}, the $H(\locring)\times{I_G}$-orbits on $\Pi_{N,r}$ parametrized by elements  $\lambda$ in  $X_{G}$ satisfying $\eqref{condition1}$ are exactly $\Pi_{\lambda,r}$.  Let $S_{n,m}$ be the set of pairs $(s, I_{s})$ such that $I_{s}$ is a subset of $\{1,\ldots, m\}$ of $n$ elements and a bijection $s: I_{s}\to\{1,\ldots, n\}$.  Let $w=(\lambda,s)\in X_{G}\times S_{n,m}$, where $\lambda=(a_1,\ldots, a_n)$ and assume $a_i<r$ for all $i$.
Denote by $\Pi_{N,r}^w$ the $I_H\times I_G$-orbit on $\Pi_{N,r}$ through the element $v$ given by
\begin{equation}
\label{def_of_IH_IG_orbit}
\left\{ \begin{array}{ll}
         v(u^{*}_i)=t^{a_{si}}e_{si} & \mbox{ for $i\in{I_s}$}\\
        v(u^{*}_i)=0 & \mbox{for $i\notin{I_s}$}.\end{array} \right. 
\end{equation}        

The closure of $\Pi_{N,r}^w$ in $\Pi_{N,r}$ will be denoted by 
$\overline{\Pi}_{N,r}^{w}$. The  simple objects of $P_{I_{H}\times I_{G}}
(\Pi(F))$ are parametrized by $X_{G}\times S_{n,m}$ \cite[Theorem 6.6]{BFH2}. For any 
element $w=(\lambda,s)$ in $X_{G}\times S_{n,m},$ the  simple object of $P_{I_{H}\times 
I_{G}}(\Pi(F))$ indexed by $w$ is the intersection cohomology sheaf $\cI^w$ of 
$\Pi_{N,r}^{w}$ for $N,r$ large enough. The object of $P_{I_H\times I_G}(\Pi(F))$ so obtained is independent of $N,r$, so our notation is unambiguous. 
We denote  $\cI^{w !}$ the objects which are extensions by zero of the constant perverse sheaf under the inclusion $\Pi_{N,r}^{w}\hook{} \Pi_{N,r}$. 
\section{Filtration and grading}
\label{filtration}
 In this section, we construct a filtration on $P_{I_H\times I_G}(\Pi(F))$ indexed by $\ZZ$. There is a natural grading of $P_{I_G}(\Fl_G)$ given by the connected components of $\Fl_G$. We will show that the filtration on $P_{I_H\times I_G}(\Pi(F))$ is compatible with the grading on $P_{I_G}(\Fl_G)$  and Hecke functor $\heckefunc{G}$. This enables us to control the action of the Hecke functors on the category $P_{I_H\times I_G}(\Pi(F))$.
 \par\medskip
%We let the affine extended Weyl group $\weyl{G}$ of $G$ act on the set $X_{G}\times S_{n,m}$ in the following way: 
%
%\begin{definition}
%\label{actionW}
%Let $w=t^{\lambda_{1}}\tau_{1}$ be an element of $\weyl{G}$ and $(\lambda,s)$ in $X_{G}\times S_{n,m}$ then we define a left action: 
%$$w\centerdot(\lambda,s)=(\lambda_{1}+\tau_{1}(\lambda),\tau_{1}s),$$
%where $\tau_{1}s$ is the composition $I_{s}\overset{s}\longrightarrow\{1,\dots,n\}\overset{\tau_{1}}\longrightarrow\{1,\dots,n\}.$
%\end{definition}
%
%This action will be used in the next chapter. The $\weyl{G}$-orbits on $X_{G}\times S_{n,m}$ are naturally in bijection with the subsets of $n$ elements in $\{1,\dots,m\}$. Namely to each orbit one associates the subset $I_{s}.$
Denote by $\pi_{1}(G)$ the algebraic fundamental group of $G$ which is formed by the elements of length zero in the affine extended Weyl group of $G$. The connected components of $Gr_G$ and $\flagvar{G}$  are indexed by $\pi_1(G)$. This yields a following natural grading 
$$\bigoplus_{\theta\in \pi_1(G)}\, P_{I_G}(\flagvar{G}^{\theta})\,\iso\, P_{I_{G}}(\flagvar{G}).$$ 
Denote by $\check{\omega}_{n}$ the character by which the group $G$ acts on  $\mathrm{det}(L_0),$ i.e. $\check\omega_{n}=(1,\dots,1).$ We may identify $\pi_1(G)$ with $\ZZ$ via the map $\theta\mapsto \langle\theta,\check{\omega}_n\rangle$. 
This grading is compatible with the convolution product on $D_{I_{G}}(\flagvar{G})$.  There is also a grading  $\iwahorihecke{G}^{k}$, $k\in\Z$ of the Iwahori-Hecke algebra $\iwahorihecke{G},$ see \cite{Lu1}. Besides the isomorphism between $K(P_{I_{G}}(\flagvar{G}))\otimes\qelbar$ and $\iwahorihecke{G}$ becomes a graded isomorphism. 
%More precisely,  for any integer $\theta$, the graded piece of index $\theta$ on the right hand side is isomorphic to the graded piece of index  $\langle\theta,\check{\omega}_n\rangle=k$ on the left hand side. 
\par\medskip
For an integer $a$ in $\Z$, let  $\Filt^{a}$ be the full subcategory in $P_{I_{H}\times{I_{G}}}(\Pi(F))$ defined as the Serre subcategory generated  by the objects $\mathcal{I}^{w},$ where $w=t^{\lambda}\tau$ are elements of $X_{G}\times S_{n,m}$  satisfying $\langle\lambda, \check{\omega}_n\rangle\geq a.$  This defines a filtration on the category $P_{I_H\times{I_G}}(\Pi(F))$ indexed by $\mathbb{Z}.$ 
\par\medskip
Let $w=t^{\lambda}\tau$ and $u=t^{\mu}\nu$ be two elements in $X_{G}\times S_{n,m}.$ The condition that $\Pi_{N,r}^{u}$ lies in the closure of $\Pi_{N,r}^{w}$ implies that 
$\langle\mu,\check{\omega}_{n}\rangle\geq \langle\lambda, \check{\omega}_{n} \rangle.$
Indeed, for any point $v$ of $\Pi_{N,r},$ the dimension of $U_{v,r}/t^{r}L$ can only decrease under specialization. For the orbit  $\Pi_{N,r}^{u}$ lying in the closure of $\Pi_{N,r}^{w}$ the number $\langle \mu , \check{\omega}_{n} \rangle$ can be arbitrary large. This number depends on $r$ is not uniformly bounded.  
\begin{lemma}
\label{compatible}
Let $w_1=t^{\lambda_{1}}\tau_{1},$ and $w_2=t^{\lambda_{2}}\tau_{2},$ be two elements in $X_{G}\times S_{n,m}$. For $i=1,2$ choose two integers $N_i$ and $r_i$  such that the following conditions are satisfied : for any $\nu\in W_{G}\lambda_{i}$

$$\hspace{1mm}\langle\nu,\check{\omega}_{1}\rangle\leq{r_1},\hspace{1mm} \langle\nu,\check{\omega}_{1}\rangle<{r_2},\hspace{1mm}\mathrm{and}\hspace{1mm}\langle -\nu,\check{\omega}_{1}\rangle\leq{N_i}.$$
 Let $v$ be an element in $\Pi_{N_{2},r_{2}}^{w_2}\subset \Pi_{\lambda_{2},r_{2}}$ and $gI_{G}$ be an element in $\flagvar{G}^{w_1}.$ For $i=1,2$ let $\mu_{i}$ be a dominant cocharacter lying in $W_{G}\lambda_{i}.$ Then  there exists a cocharacter $\mu$ smaller than or equal to $\mu_{1}+\mu_2$  such that $gv$ belongs to $\Pi_{\mu,r_1+r_2}.$ 
\end{lemma}
\begin{proof}
The lattice $U_{v,r_{2}}=v(U^{*})+t^{r_{2}}L$ lies in $O^{\lambda_{2}}\subset Gr_{G}^{\mu_{2}}.$ Thus $g(U_{v,r_{2}})$ lies in $gO^{\lambda_{2}}$. Since $gG(\locring)\in O^{\lambda_{1}}\subset Gr_{G}^{\mu_{1}}$, we have 
$$I_{G}t^{\lambda_{1}}G(\locring)t^{\lambda_{2}}G(\locring)/ G(\locring)\subset \overline {Gr}_{G}^{\mu_{1}+\mu_{2}}$$
and this implies the assertion. 
\end{proof}
\begin{proposition}
Let $w_1=t^{\lambda_{1}}\tau_{1},$ and $w_2=t^{\lambda_{2}}\tau_{2},$ be two elements in $X_{G}\times S_{n,m}$. For $i=1,2$ choose two integers $N_i$ and $r_i$  such that the following conditions are satisfied : for any $\nu\in W_{G}\lambda_{i}$

$$\hspace{1mm}\langle\nu,\check{\omega}_{1}\rangle\leq{r_1},\hspace{1mm} \langle\nu,\check{\omega}_{1}\rangle<{r_2},\hspace{1mm}\mathrm{and}\hspace{1mm}\langle -\nu,\check{\omega}_{1}\rangle\leq{N_i}.$$
 Then $\heckefunc{G}(L_{w_{1}},\mcI^{w_{2}})$ lies in $\Filt^{d}$ with $d=\langle \lambda_{1}+\lambda_{2},\check{\omega}_{n}\rangle.$
\end{proposition}

\begin{proof}
The sheaf $\mathcal{I}^{w_{2}}$ is the IC-sheaf of the orbit $\Pi_{N,r}^{w_{2}}$ which is a subscheme of $\Pi_{\lambda_{2},r_{2}}.$ In the notation of Appendix $\ref{Appendix}$,  we have $\overline{\flagvar{G}^{w_{1}}}\subset {}_{{}_{r_{1},N_{1}}}\mathcal{F}l_{G}.$  Choose $r\geq r_{1}+r_{2}$ and $s\geq N_{2}+r_{2}.$   The space 
$\Pi_{N_{2},r_{2}}\newtimes \overline{\flagvar{G}^{w_{1}}}$ is the scheme classifying pairs $(v,gI_{G})$, where  $gI_{G}$ is in $\overline{\flagvar{G}^{w_{1}}}$ and $v$ is in $t^{-N_{2}}g\Pi/t^{r}\Pi$. We have the following diagram 
$$\Pi_{N_{1}+N_{2},r}\overset{\pi}\longleftarrow \Pi_{N_{2},r_{2}}\newtimes \overline{\flagvar{G}^{w_{1}}}\overset{act_{q,s}}{\longrightarrow} K_{s}\backslash (\Pi_{N_{2},r_{2}}),$$ 
where $\pi$ is the projection sending $(v,gI_{G})$ to $v$. Let $\mcI^{w_{2}}\bt L_{w_{1}}$ be the twisted exterior product of $\mcI^{w_{2}}$ and $L_{w_{1}}$ over $\Pi_{N_{2},r_{2}}\newtimes \overline{\flagvar{G}^{w_{1}}}$ which is normalized to be perverse. Then by definition
$$\heckefunc{G}(L_{w_{1}},\mcI^{w_{2}}) \iso \pi_{!}(\mcI^{w_{2}}\bt L_{w_{1}}).$$
In our case $\mcI^{w_{2}}\bt L_{w_{1}}$ is  the IC-sheaf of $act_{q,s}^{-1}(\overline{\Pi}_{N_{2},r_{2}}^{w_{2}}).$
For a point $v$ in $\Pi_{N_{1}+N_{2},r}$, let $\mu$ be in $X_{G}$ such that $U_{v,r}$ lies in $O^{\mu}.$ The part of the fibre of the map $\pi$ over $v$ that contributes to $\pi_{!}(\mcI^{w_{2}}\bt L_{w_{1}})$ is 
$$\{ gI_{G}\in \overline{\flagvar{G}^{w_{1}}}\vert g^{-1}v\in \overline{\Pi}_{N_{2},r_{2}}^{w_{2}}\}.$$
The latter scheme is empty unless $\langle \mu, \check{\omega}_{n}\rangle\geq \langle \lambda_{1}+\lambda_{2},\check{\omega}_{n}\rangle.$ It follows that $\heckefunc{G}(L_{w_{1}},\mcI^{w_{2}})$ lies in $\Filt^{d}$ with $d=\langle \lambda_{1}+\lambda_{2},\check{\omega}_{n}\rangle.$
\end{proof}

\begin{theorem}
\label{filt}
Assume $n\leq m.$ The filtration $\Filt^d$ on $P_{I_H\times I_G}(\Pi(F))$ is compatible with the grading on $P_{I_G}(\flagvar{G})$ defined by the connected components. Namely set $w_1=t^{\lambda_{1}}\tau$ with $\tau$ in $S_{n,m},$ $\lambda_{1}$ in $X_{G},$ and let  $m_1=\langle\lambda_1,\check{\omega}_n\rangle.$ Then $\heckefunc{G}(L_{w_1},.)$ sends an irreducible object of $\Filt^d$ to a direct sum of shifted objects of $\Filt^{d+m_1}$.
\end{theorem}

\begin{proof}
We use the notation of Lemma $\ref{compatible}.$ For $gI_{G}$ in $\overline{\flagvar{G}^{w_{1}}},$ let $L^{'}=gL$ and equip $L^{'}$ with the flag $L_{i}^{'}=gL_{i}$ for $i=1,\dots,n$.  Let $v$ be the map from $U^{*}$ to $t^{-N_{2}}L^{'}/ t^{r}L$ such that its composition with 
$$t^{-N_{2}}L^{'}/ t^{r}L\longrightarrow t^{-N_{2}}L^{'}/ t^{r_{2}}L^{'}$$
lies in the closure of the orbit $(U_{N_{2}r_{2}}\otimes L^{'}) ^{w_{2}}.$ The latter scheme is the corresponding orbit on $U_{N_{2},r_{2}}\otimes L^{'}.$ The relative dimension formula  gives us
$$\dim(L,L^{'})+\dim(L^{'},v(U^{*})+t^{r_{2}}L^{'})=\dim(L,v(U^{*})+t^{r_{2}}L^{'}).$$
Moreover, we have
$$\dim(L^{'},v(U^{*}+t^{r_{2}}L^{'}))\geq \langle \lambda_{2},\check{\omega}_{n} \rangle $$
and 
$$\dim(L,L^{'})=\langle\lambda_{1},\check{\omega}_{n}\rangle.$$
This leads to $\dim(L,v(U^{*}+t^{r_{2}}L^{'}))\geq \langle\lambda_{1}+\lambda_{2},\check{\omega}_{n}\rangle.$ On the other hand we have $t^{r}L\subset t^{r_{2}}L^{'}$ so 
$\dim(L,v(U^{*})+t^{r}L)$ can not be strictly smaller than  $\dim(L,v(U^{*})+t^{r_{2}}L^{'}).$ 
\end{proof}
As a consequence of this theorem, $K(P_{I_{H}\times I_{G}}(\Pi(F)))$ is a filtered module over $\iwahorihecke{G},$ so that graded part $\oplus_{d\in\Z}\Filt^{d}/ \Filt^{d+1}$ is a left $\iwahorihecke{G}$-module. 
\section{Kudla's filtration and some submodules}
 \subsection{Case \texorpdfstring{$n=m$}{n=m}}$ $ 
\label{nm}

In this subsection we will assume $n=m$. We will show that the submodule in $K(D_{I_H\times I_G}(\Pi(F)))$ generated by $\cI^{w_{0}!}$ over $K(P_{I_G}(\Fl_G))$ is free of rank one (and is also preserved by the action of  $P_{I_H}(\Fl_H)$). We also precise an equivalence of categories $\tilde\sigma: P_{I_H}(\Fl_H)\,\isom\, P_{I_G}(\Fl_G)$ which defines at the level of functions an anti-involution of Iwahori-Hecke algebras $\iwahorihecke{G}$ and $\iwahorihecke{H}$. By means of this equivalence we relate the actions of Hecke functors for $H$ and $G$ on the submodule generated by $\cI^{w_{0}}$ (resp. $\cI^{w_{0}!}$). 
\par\medskip
Denote by $w_{0}$ the longest element of $W_{G}.$   Let $\mcI^{w_{0}}$ be the IC-sheaf of the orbit $\Pi_{0,1}^{w_{0}}.$ Let $\Pi_{\mcI^{w_{0}}}$ be subscheme of $\Hom(U_0^*, L_0)$ consisting of elements $v$ such that $v$ sends $\Vect(u_n^*, \ldots, u_{n-i+1}^*)$ to 
$L_i=\Vect(e_1,\dots,e_i),$ for $i=1,\dots,n$. Note that $\Pi_{\mcI^{w_{0}}}$ is an affine space, the closure of $\Pi_{0,1}^{w_{0}}.$ Thus $\mcI^{w_{0}}$ is the constant perverse sheaf on $\Pi_{\mcI^{w_{0}}}.$ For any $w$ in the affine extended Weyl group $\weyl{G},$ let us describe $\heckefunc{G}(L_w, \mcI^{w_{0}}).$ Let $w=t^{\lambda}\tau,$ where $\tau\in W_{G}$ and $\lambda\in X_{G}.$  Let $N,r$ be two integers with $N+r\geq 0$  such that the following condition is verified: for any $\nu$ in $W_{G}.\lambda$ we have
\begin{equation}
\label{condition2}
\langle \nu,\check{\omega}_{1}\rangle< r \hspace{2mm}\mathrm{and}\hspace{2mm} \langle -\nu,\check{\omega}_{1} \rangle \leq N. 
\end{equation}
For any element $gI_G$  of  $\overline{\flagvar{G}^{w}}$, we put $L'=gL$ and  we equip $L'/tL'$ with the flag $L'_{i}=gL_{i},$ for $i=1,\dots,n.$ 
Let $\Pi_{\mathcal{I}^{w_{0}},r}\tilde{\times}\overline{\flagvar{G}^{w}}$ be the scheme classifying pairs $(v,gI_{G}),$ where $gI_G$ is in $\overline{\flagvar{G}^{w}}$ and  $v$ is a map $U^*\to L'/t^{r}L$ such that the induced map
\begin{equation}
\label{2}
\overline{v}:U^{*}/tU^{*}\longrightarrow L'/tL'
\end{equation} 
sends $\Vect(u_n^*,\dots,u_{n-i}^*)$ to $L'_{i+1},$ for $i=0,\dots,n-1.$  Let 
$$
\pi:\Pi_{\mathcal{I}^{w_{0}},r}\tilde{\times}\overline{\flagvar{G}^{w}}\longrightarrow\Pi_{N,r}
$$
be the map sending $(v, gI_G)$ to $v$.
The second projection $pr:\Pi_{\mathcal{I},r}\newtimes\overline{\flagvar{G}^{w}}\longrightarrow\overline{\flagvar{G}^{w}}$ is a locally trivial fibration with fibres isomorphic to the affine space. Let $\qelbar\bt L_w$ be the perverse sheaf  $pr^*L_w[\dimrel(pr)]$. 
Then by definition 
$$\heckefunc{G}(L_{w},\mcI^{w_{0}})\iso \pi_{!}(\qelbar\bt L_w).$$
Remark that the condition $\eqref{condition2}$ initially appears in the construction of the irreducible objects $\mathcal{I}^{w}$ in $P_{I_{H}\times I_{G}}(\Pi(F))$ in \cite{BFH2}.  
\begin{lemma}
For any $w\in\wt W_G$ the perverse sheaf $\cI^{ww_0}$ appears in $\heckefunc{G}(L_w,\mcI^{w_0})$ with multiplicity one.
\end{lemma}
\begin{proof}
 Consider the open subscheme $V$ in $\Pi_{\mathcal{I},r}\newtimes\overline{\flagvar{G}^{w}}$ given by the conditions that $gI\in \Fl^w_G$, and that the map $\eqref{2}$ is surjective. Clearly, $\pi(V)$ is contained in $\Pi_{N,r}^{ww_0}$. So, $\pi$ can be viewed as a map
$$
\pi:\Pi_{\mathcal{I},r}\tilde{\times}\overline{\flagvar{G}^{w}}\longrightarrow \bar \Pi_{N,r}^{ww_0}
$$

The restriction of the complex $\heckefunc{G}(L_w,\mcI^{w_{0}})$  to $\Pi_{\lambda,r}\subset \Pi_{N,r}$ identifies with $\mathcal{I}^{ww_0}$.
\end{proof}

\begin{proposition}
\label{Lw!Iw0}
For any $w$ in $\weyl{G}$ there is a canonical isomorphism $\heckefunc{G}(L_{w!},\mathcal{I}^{w_{0}!})\,\isom\, \mathcal{I}^{ww_{0}!}.$ 
\end{proposition}

\begin{proof}
Let $\Pi_{\mathcal{I}^{w_{0}},r}^{0}\newtimes \flagvar{G}^{w}\subset \Pi_{\mathcal{I}^{w_{0}},r}\newtimes \overline{\flagvar{G}^{w}}$ be the open subscheme of pairs  $(v,gI_{G})$  such that  $gI_G$ is in  $\flagvar{G}^{w}$
and the map $\overline{v}$ in $\eqref{2}$  is an isomorphism. For any points $(v,gI_{G})$ in this subscheme the map $\overline{v}$ is an isomorphism between $\Vect(u^*_1,\dots,u^*_{n-i})$ and $L'_{i+1},$ for $i=0,\dots,n-1.$ Let 
$$
\pi^{0}:\Pi_{\mathcal{I}^{w_{0}},r}^{0}\newtimes \flagvar{G}^{w}\longrightarrow \Pi_{N,r}
$$
be the restriction of $\pi$ to $\Pi_{\mathcal{I}^{w_{0}},r}^{0}\newtimes \flagvar{G}^{w}$.
Thus by definition we have 
$$
\heckefunc{G}(L_{w!},\mathcal{I}^{w_{0}!})\iso {\pi^{0}_{!}(\qelbar\bt L_w)}
$$
The image of $\pi^{0}$ is equal to  $\Pi_{N,r}^{ww_0}$ and the map $\pi^{0}$ is an isomorphism onto its image.
\end{proof}
\begin{definition}
\label{antiauto}
For $\lambda$ in $X_{G}$ and $\tau$ in $W_{G},$ let $w\longrightarrow \overline{w}$ be the map  from $\weyl{G}$ to  $\weyl{G}$  defined by 
$$\overline{t^{\lambda}\tau} \longrightarrow t^{\tau^{-1}(\lambda)} \tau^{-1}$$
This is an anti-automorphism of $\weyl{G}$. Note that $\overline{w_{0}}=w_{0}.$
\end{definition}
The following analog of Proposition~\ref{Lw!Iw0} for $H$ instead of $G$ is proved similarly.
  \begin{proposition}
\label{Lw!w0G}
For $w$ in $\weyl{H}$, the complex $\heckefunc{H}(L_{w!},\mcI^{w_{0}!})$ is canonically isomorphic to $\mcI^{\overline{ww_{0}}!}.$ The sheaf $\mcI^{\overline{ww_{0}}}$ occurs in $\heckefunc{H}(L_{w},\mcI^{w_{0}})$ with multiplicity one. $\square$
\end{proposition}
The following example shows that $\heckefunc{G}(L_w,\mcI^{w_{0}})$ is not always irreducible and gives us some interesting objects in $D_{I_{H}\times I_{G}}(\Pi(F))$.
 \begin{corollary} Let $1\le i<n$. 
Let $w$ be the transposition $(i,i+1)$ in  $W_{G},$ $\lambda=(0,\dots,0,1,0,\dots,0),$ where $1$ appears on the $i^{th}$ position, and  $w'=t^{\lambda}w_0.$ Then
$$\heckefunc{G}(L_w,\mcI^{w_{0}})\iso \mathcal{I}^{ww_0}\oplus \mathcal{I}^{w^{'}}.$$
\end{corollary}

\begin{proof} 
The variety $\overline{\flagvar{G}^{w}}$ classifies lattices $L_{0}^{'}$ endowed with a complete flag of lattices $L^{'}_{-1}\subset L^{'}_{0} \subset L_{1}^{'}\subset \dots$ such that $L_{i+n}'=t^{-1}L_i$ for all $i$, and 
$L^{'}_{j}=L_{j}$ unless $j=i$ mod $n$. Here $L_{j}$ is the standard flag on $L(F)$. So, $\overline{\flagvar{G}^{w}}$ identifies with the projective space of lines in $L_{i+1}/L_{i-1}.$ Let $Y_i$ be the  closed subscheme of $\Pi_{0,1}$ given by $v(W_j)\subset{L_j}$ for $j\neq i$. Here $\{W_j\}$ is the flag on $U_0^*$ preserved by $B_H$. Note that $Y_{i}$ is an affine space. 
Define a closed subscheme  $Y_{i}^{'}$ of $Y_{i}$ consisting of elements $v$ of $Y_{i}$ such that  $v(L_i)\subset L_{i-1}.$ Then $Y_{i}^{'}$ is also an affine space. 
 \par\medskip
Let $\Pi_{\mathcal{I}}\newtimes\overline{\flagvar{G}^{w}}$ be the scheme classifying pairs $(v,gI_{G}),$ where $gI_{G}$ is in $\overline{\flagvar{G}^{w}}$ and $v: U_0^*\to L_{0}$ such that $v(W_{j})\subset gL_j,$ for all $1\leq j\leq n.$ 
 We have the diagram 
$$
Y_i
\overset{\pi}{\longleftarrow}\Pi_{\mathcal{I}}\newtimes\overline{\flagvar{G}^{w}}\overset{pr}{\longrightarrow}\overline{\flagvar{G}^{w}}.$$
By definition of the Hecke operators one has 
$$\heckefunc{G}(L_w,\mcI^{w_{0}})=\pi_{!}(\qelbar\bt L_{w}).$$
For a point $v$ in $Y_{i}\backslash Y_{i}^{'}$ the fibre of the map $\pi$ over $v$  is reduced to a point and the map $\pi$ is an isomorphism over $Y_{i}\backslash Y_{i}^{'}$. The restriction of $\heckefunc{G}(L_w,\mcI^{w_{0}})$ to $Y_{i}\backslash Y_{i}^{'}$ is isomorphic to $\IC(Y_i)=\qelbar[\dim(Y_i)].$ On the other hand, the space $Y_{i}$ identifies with $\Pi_{N,r}^{ww_{0}}$. The fibre of $\pi$ over a point $v$ of $Y_{i}^{'}$ is isomorphic to $\mathbb{P}^1$. Since $Y'_i$ is an affine space of codimension 2 in $Y_i$, we are done. 
\end{proof} 

Proposition $\ref{filt}$ about the filtration in the special case of $\mathcal{I}^{w_{0}}$ yields the following:

\begin{corollary}
Let $w=t^{\lambda}\tau,$ where  $\lambda$ is in $X_{G}$ and $\tau$ is in $W_{G}$. Then if  $d=\langle\lambda,\check{\omega}_n\rangle,$  there exists $K$ in  $\Filt^{d+1}$ such that $\heckefunc{G}(L_{w},\mathcal{I}^{w_{0}}),\isom\, \mathcal{I}^{ww_0}\oplus K.$ $\square$
\end{corollary}
It also follows that for $n=m,$ the space $\oplus_{d\in\Z}\Filt^{d}/ \Filt^{d+1}$ is a free module of rank one over $\iwahorihecke{G}$ generated by $\mcI^{w_{0}}$.
Note that the  homomorphism of $\mathcal{H}_{I_{G}}$-algebras $\mathcal{H}_{I_{G}} \longrightarrow \mathcal{S} (\Pi(F))^{I_{H}\times I_{G}}$ sending 
$\mathcal{S}$ to  $\heckefunc{G}(\mathcal{S},\mathcal{I}^{w_{0}!})$
is injective. The submodule generated by $\mathcal{I}^{w_{0}!}$ is a free  module of rank one over each one of the Iwahori-Hecke algebra $\iwahorihecke{G}$ and $\iwahorihecke{H}.$  
\par\bigskip
We may stratify $\Pi_{N,r}$ in a slightly different way. Let $\theta$ be any element of $\pi_{1}(G)$ and let  
$\lambda$ be a lift of $\theta$ in $X_{G}$ satisfying condition $\eqref{condition2}$.
We define a locally closed subscheme  $\Pi_{N,r}^{\theta}$  of $\Pi_{N,r}$ as follows:  
$$\Pi_{N,r}^{\theta}=\{v:U^*\to t^{-N}L/t^{r}L\hspace{1mm}\textrm{such}\hspace{1mm} \textrm{that}\hspace{1mm} \dim(U_{v,r}/t^{r}L)=\dim (t^{\lambda}L/t^{r}L)\}.$$
This definition is in fact independent of the lift $\lambda$.
\par\medskip
 For a given $w=t^{\lambda}\tau$ in $\weyl{G}$, let $\theta$ be the image of $\lambda$ in $\pi_{1}(G).$ Let $\tilde{\mcI^{w}}$ be the extension by zero of $\mcI^{w}\vert_{\Pi_{N,r}^{\theta}}$ on $\Pi_{N,r}.$ Then we have the following result: 

\begin{proposition}
\label{tildew}
For any $w$ in $\weyl{G}$, we have two canonical isomorphisms
$$\heckefunc{G}(L_{w},\mathcal{I}^{w_{0}!})\iso \tilde{\mcI}^{ww_{0}}
\hspace{5mm} \heckefunc{H}(L_{w},\mathcal{I}^{w_{0}!})\iso \tilde{\mcI}^{\overline{ww_{0}}},$$
 where the anti-involution $w\to \overline{w}$ is defined in Definition $\ref{antiauto}.$
\end{proposition}

\begin{proof}
We show the assertion for $\heckefunc{G}$, the case of  $\heckefunc{H}$ may be proved similarly. As in Proposition $\ref{Lw!Iw0}$, consider the open subscheme $\Pi_{\mathcal{I}^{w_{0},r}}^{0} \newtimes \flagvar{G}^{w}$ inside $\Pi_{\mathcal{I}^{w_{0}},r}\newtimes \overline{\flagvar{G}^{w}}$ given by the additional condition that the map 
$$\overline{v}:U^{*}/tU^{*}\to L^{'}/tL^{'}$$
is an isomorphism. Thus the restriction 
$$\pi^{0}:\Pi_{\mathcal{I}^{w_{0},r}}^{0} \newtimes \flagvar{G}^{w}\longrightarrow \Pi_{N,r}$$
of the map $\pi$ is locally a closed immersion. Therefore by definition 
$$\heckefunc{G}(L_{w},\mcI^{w_{0}!})\iso \pi^{0}_{!}(\qelbar\bt L_{w})\iso \tilde{\mcI}^{ww_{0}}.$$
\end{proof}
\par\medskip
The map  from $G(F)$ to $G(F)$  sending $g$ an element of $G(F)$ to $g^{-1}$ induces an equivalence of categories
$$\star^{\sharp}:P_{I_{G}}(\flagvar{G})\iso P_{I_{G}}(\flagvar{G}).$$
The similar equivalence of categories holds for $P_{\I{H}}(\flagvar{H}).$ Hence for $w$ in the affine extended Weyl group,  we have canonical isomorphisms 
$$\star^{\sharp}(L_{w})\iso L_{w^{-1}}, \hspace{5mm}\star^{\sharp}(L_{w!})\iso L_{w^{-1}!}, \hspace{5mm}\star^{\sharp}(L_{w*})\iso L_{w^{-1}*}.$$
At the level of Iwahori-Hecke algebras $\star^{\sharp}:\iwahorihecke{G}\to \iwahorihecke{G}$ is an anti-isomorphism of algebras. 
\begin{definition}
\label{rla}
Assume that $n\leq m.$ For any $\mathcal{T}$ in $P_{\I{H}}(\flagvar{H})$ and $\mathcal{K}$ in $P_{I_{H}\times \I{G}}(\Pi(F))$, we  define the right action functor $\heckefuncr{H}(\mathcal{T},\mathcal{K})$ of $P_{I_{H}}(\flagvar{H})$ on $P_{I_{H}\times I_{G}}(\Pi(F))$:  
$$\heckefuncr{H}(\mathcal{T},\mathcal{K})=\heckefunc{H}(\star^{\sharp}(\mathcal{T}),\mathcal{K}). $$
\end{definition}
The anti-automorphism defined over $\weyl{G}$ in Definition $\ref{antiauto}$ sending any $w$ to $\overline{w}$ may be extended to  an anti-automorphism of the group $G(F)$ itself. It suffices to take the morphism sending any $g$ an element of  $G(F)$ to its transpose ${}^t \! g$. Denote by $\sigma$ the anti-involution defined over $G(F)$ sending $g$ to $w_{0} {}^t \! g w_{0}.$ This anti-involution preserves the Iwahori subgroup $I_{G}$ and  induces an equivalence of categories (still denoted by $\sigma$):
$$\sigma:P_{I_{G}}(\flagvar{G})\iso P_{\I{G}}(\flagvar{G}).$$ 
Remind that $n=m$. We have the following result:

\begin{theorem}
\label{sig}
 There exists an equivalence of categories  
\begin{align}
\sigma : P_{I_{G}}(\flagvar{G})& \iso P_{I_{H}}(\flagvar{H})\nonumber \\
L_{w}&\longrightarrow  L_{w_{0}\overline{w} w_{0}},
\end{align}
Additionally it verifies the following properties:  for any $w$ and $w^{'}$ in $\weyl{G}$ we have 
$$\heckefunc{G}(L_{w},\mcI^{w_{0}!})\iso\heckefunc{H}(\sigma(L_{w}),\mcI^{w_{0}!})$$
and 
$$\sigma(L_{w}\star L_{w^{'}})=\sigma (L_{w^{'}})\star \sigma(L_{w}),$$
where $\star$ is the convolution product in $P_{I_{G}}(\flagvar{G})$. 

\end{theorem}

\begin{proof}

The assertion follows from Propositions $\ref{Lw!Iw0}$  and  $\ref{Lw!w0G}$ and  $\ref{tildew}.$ 
\end{proof}
The two anti-isomorphisms $\sigma$ and $\star^{\sharp}$ defined above commute and their composition is an algebra isomorphism.   We will denote this composition by $\tilde{\sigma},$ i.e. for any $g$ in $G(F)$,  $\tilde{\sigma}(g)=w_{0}{}^t \! g^{-1}w_{0}.$

% Let $\tilde\sigma: G(F)\to G(F)$ be the involution sending $g$ to $w_0(^tg^{-1})w_0$. This $\tilde\sigma$ is an automorphism preserving $I_G$, so it induces a covariant equivalence of categories 
%\begin{equation} 
%\label{equivalence_tilde_sigma*}
% \tilde\sigma^*: P_{I_G}(\Fl_G)\,\isom\, P_{I_G}(\Fl_G)
%\end{equation} 
%Moreover, one has naturally $\tilde\sigma^*(\cT_1\ast\cT_2)\,\iso\, \tilde\sigma^*(\cT_1)\ast \tilde\sigma^*(\cT_2)$ for $\cT_i\in P_{I_G}(\Fl_G)$. 
%If $w\in \wt W_G$ then $\tilde\sigma^*(L_{w!})\,\iso\, L_{\nu !}$ with $\nu=w_0\bar w^{-1}w_0$. 
%
%Propositions~\ref{Lw!Iw0} and \ref{Lw!w0G} imply the following.
%
%\begin{proposition} 
%\label{sig}
%Remind our assumption $n=m$. For any $\cT\in D_{I_G}(\Fl_G)$ of the form $L_{w!}$ for some $w\in \wt W_G$, there is a canonical isomorphism 
%$\heckefunc{G}(\tilde\sigma^*(\cT), \cI^!)\,\isom\, \heckefuncr{H}(\cT, \cI^!)$.
%\end{proposition}
%
% One may ask if this isomorphism still holds for any $\cT\in D_{I_G}(\Fl_G)$. 

\subsection{Sub-modules \texorpdfstring{$\Theta$}{Theta} and \texorpdfstring{$\mathcal{S}_{0}$}{\mathcal{S}0}}$ $
\label{theta}

We assume in this subsection that $n\leq m$ and we consider the module $\Theta$  generated by the elements $\mathcal{I}^{\mu!},$ where $I_{\mu}$ runs through all possible subsets of $n$ elements in $\{1,\dots ,m\}$. This is a submodule in the Grothendieck group $K(D_{I_H\times I_G}(\Pi(F)))$ of $D_{I_H\times I_G}(\Pi(F)))$ acted on by  $K(P_{I_G}(\Fl_G))$.  We will show that the submodule $\Theta$ is free of rank $C_{m}^{n}$ over $K(P_{I_G}(\Fl_G))$ with an explicit basis formed by $\mathcal{I}^{\mu!}$. The submodule $\Theta$ is a key object in the proof of the classical Howe correspondence \cite{Minguez1} and \cite{Kudla}. It is indeed the first term of the Kudla's filtration defined over the Weil representation in \cite{MVW}. The considerations in this subsection are essentially on the level of Grothendieck groups, we formulate them on the level of derived categories however when this is possible.
\par\medskip 
Let $\cS_{0}$ be the $\Qlb$-subspace  of  $K(D_{I_H\times I_G}(\Pi(F))\otimes\Qlb$ generated by elements of the form $\cI^{(w\centerdot w_0)!}$, where  $w$ runs through $\widetilde{W}_G$ the affine extended Weyl group of $G$ and the action $w\centerdot w_0$ is defined below. The space  $\cS_0$  is a free module  of rank one over $\cH_{I_G}$.   We consider the standard Levi subgroup $M$ of $H$ corresponding to the partition $(n,m-n)$ of $m$ and we recall briefly the construction of the subalgebra $\iwahorihecke{M}$  of the Iwahori-Hecke algebra $\iwahorihecke{H}$ and some properties according to \cite{Prasad}. Then we endow $\mathcal{S}_{0}$ with a right action of $\iwahorihecke{M}$ and by parabolic induction we construct an induced module.  We show that  the adjunction map $\alpha:\mathcal{S}_{0}\otimes_{\iwahorihecke{M}} \iwahorihecke{H} \to \mathcal{S}^{I_{H}\times I_{G}}(\Pi(F))$ is injective and its image equals $\Theta\otimes \qelbar.$ This gives the first therm of Kudla's filtration as an induced module. In the rest of this subsection we show that the action of the Iwahori-Hecke algebra of the factor  $\GL_{n}$ of $M$ identifies  with the action of $\iwahorihecke{G}$ via the anti-involution $\tilde{\sigma}$ defined in Theorem $\ref{sig}$. The action of the Iwahori-Hecke algebra of $\GL_{m-n}$ factor of $M$ is  by shifting by $[-\ell(w)]$, where $\ell$ denotes the length function on $\weyl{G}$.  
\par\medskip

We let the affine extended Weyl group $\weyl{G}$ of $G$ act on the set $X_{G}\times S_{n,m}$ in the following way: 
\begin{definition}
\label{actionW}
Let $w=t^{\lambda_{1}}\tau_{1}$ be an element of $\weyl{G}$ and $(\lambda,s)$ in $X_{G}\times S_{n,m}$. We define a left action: 
$$w\centerdot(\lambda,s)=(\lambda_{1}+\tau_{1}(\lambda),\tau_{1}s),$$
where $\tau_{1}s$ is the composition $I_{s}\overset{s}\longrightarrow\{1,\dots,n\}\overset{\tau_{1}}\longrightarrow\{1,\dots,n\}.$
\end{definition}
We will consider the affine extended Weyl group $\weyl{G}$ as a subset of $X_{G}\times S_{n,m}.$ More precisely to a given $w=t^{\lambda}\tau$  we associate the element $(\lambda,\tau)$ in $X_{G}\times S_{n,m}$ with $I_{\tau}=\{1,\dots,n\}.$ Let $I_{w_{0}}=\{1,\dots,n\}$  be a subset of $\{1,\dots ,m\}$ and $w_{0}:I_{w_{0}}\to \{1,\dots,n\}$ be the longest element of  the Weyl group $W_{G}$. By the above convention the element $w_{0}$ becomes the element $(0,w_{0})$ in $X_{G}\times S_{n,m}.$
\par\medskip
For any   strictly decreasing map $\nu$ from $\{1,\dots,n\}$ to $\{1,\dots,m\}$, denote by $I_{\mu}$ the image  of $\nu$ and denote by  $\mu:I_{\mu}\to \{1,\dots,n\}$ the inverse of $\nu.$ Thus $\mu$ can be viewed as an element of $X_{G}\times S_{n,m}$ by assuming that the corresponding term on $X_{G}$ vanishes.
Let  $\overline{\Pi}_{0,1}^{\mu}$ be the closure of $I_{H}\times I_{G}$-orbit $\Pi_{0,1}^{\mu}$ in $\Pi_{N,r}$, $\overline{\Pi}_{0,1}^{\mu}$ is an affine space. Denote by $\mcI^{\mu}$ the IC-sheaf of $\Pi_{0,1}^{\mu},$  it is the constant perverse sheaf on its support. 
\par\medskip
Denote  by $U_{1}\subset U_{2}\subset \dots \subset U_{m}=U_{0}$ the standard flag on $U/tU.$
We  consider  $\Pi_{0,1}$  the space of maps $v:L^*\to U/tU$ such that the domain and the range are both equipped with a flag preserved by $v$. Thus $\overline{\Pi}_{0,1}^{\mu}$ is the space of maps $v:L^{*}\to U/tU$ such that $v(e^{*}_{i})$ lies in $U_{\nu(i)}$ for all $i=1,\dots,n.$ 
In other terms the map $v$ sends $\Vect(e_{n}^{*},\dots,e_{n-i}^{*})$ to $U_{\nu(n-i)}$ for all $i=0,\dots,n-1.$ An element $v$ lies in $\Pi_{0,1}^{\mu}$ if  additionally the map
 sending $\Vect (e_{n}^{*},\dots,e_{n-i}^{*})$ to $U_{\nu(n-i)}/U_{\nu(n-i)-1}$ is non-zero for all $i=1,\dots,n.$  We  may also consider the element $v$ in $\Pi_{0,1}$ as a map from $U^*$ to $L/tL,$ so $v$ lies in $\overline{\Pi}_{0,1}^{\mu}$ if and only if $v$ sends $\Vect (u^{*}_{m},\dots,u^{*}_{1+\nu(j)})$ to $L_{j-1}$ for all $j=1,\dots,n.$ Moreover, the map  $v$ lies in  $\Pi_{0,1}^{\mu}$ if in addition  $v(u^{*}_{\nu(j)})\notin L_{j-1}$  for all $j=1,\dots,n.$ Let $w=t^{\lambda}\tau$ be  an element of $\weyl{G}.$ Choose  two integers $N$ and $r$ with $N+r> 0$ such that for any $\nu$ in $W_{G}.\lambda$ the following condition is satisfied (condition $\eqref{condition2}$):
 $$\langle \nu,\check{\omega}_{n} \rangle<r\hspace{5mm}\mathrm{and} \hspace{5mm}\langle-\nu,\check{\omega}_{n}\rangle\leq N.$$ For a point $gI_{G}$ in $\overline{\flagvar{G}^{w}}$, we set $L^{'}=gL$ and equip $L^{'}/tL^{'}$ with the complete flag $L^{'}_{i}=gL_{i}$ for $i=1,\dots ,n$. Here $(L_1\subset\ldots\subset L_n=L/tL)$ is the complete flag on $L/tL$ preserved by $B_G$. Let $\overline{\Pi}_{r}^{\mu}\newtimes \overline{\flagvar{G}^{w}}$ be the scheme classifying pairs $(v,gI_{G}),$ where $gI_{G}$ is in $\overline{\flagvar{G}^{w}}$, and $v:U^{*}\to L^{'}/t^{r}L$ such that the induced map 
$$
\overline{v}:U^{*}/tU^{*}\longrightarrow L^{'}/tL^{'}
$$
sends $\Vect (u_{m}^{*},\dots ,u_{\nu(j)+1}^{*})$ to $L_{j-1}^{'}$ for all $j=1,\dots, n.$ We have a proper map 
$$
\pi:\overline{\Pi}_{r}^{\mu}\newtimes \overline{\flagvar{G}^{w}}\longrightarrow \Pi_{N,r}
$$
sending any element  $(v,gI_{G})$ to $v$. By definition of $\heckefunc{G}$,
$$
\heckefunc{G}(L_{w},\mathcal{I}^{\mu})\iso \pi_{!}(\qelbar\bt L_{w}),
$$
where $\qelbar\bt L_{w}$ is normalized to be perverse. 
 \begin{proposition}
\label{grothen4}
Let $w$ be an element of $\weyl{G}$. Then  $\heckefunc{G}(L_{w!},\mathcal{I}^{\mu!})$ is canonically isomorphic to $\mcI^{w\centerdot\mu !}.$

\end{proposition}

\begin{proof}
Let $w=t^{\lambda}\tau$  with $\lambda=(a_{1},\dots,a_{n})$ in $X_{G}$ and $\tau$ in $W_{G}.$
Let $\Pi_{r}^{\mu}\newtimes \flagvar{G}^{w}$ be  the open subscheme  of  $\overline{\Pi}_{r}^{\mu}\newtimes \overline{\flagvar{G}^{w}}$ given by the additional conditions that  $gI_{G}$ lies in $\flagvar{G}^{w}$ and that the map $\overline{v}:\Vect (u_{m}^{*},\dots ,u_{\nu(j)}^{*})\to L_{j}^{'}$
is surjective for $j=1,\dots,n.$ Denote by $\pi^{0}$ the restriction of $\pi$ to this open subscheme. The image of $\pi^{0}$ consist of  the $I_{H}\times I_{G}$-orbit on $\Pi_{N,r}$ through $v$ such that $v(u^{*}_{\nu(j}))=t^{a_{\tau(j)}}e_{\tau(j)}$ for all $j=1,\dots,n$ and $v(u^{*}_{k})=0$ for $k\in I_{\mu}.$
Therefore the image of the map $\pi^{0}$  is $\Pi_{N,r}^{w\centerdot\mu}$ and $\pi^{0}$  is an isomorphism onto its image.  Thus 
$$\heckefunc{G}(L_{w!},\mathcal{I}^{\mu!})\iso \mathcal{I}^{(w\centerdot \mu)!}.$$ 
\end{proof}

\begin{definition}
Let $\Theta$ be  the $K(P_{I_G}(\Fl_G))$-module in $K(D_{I_H\times I_G}(\Pi(F)))$ generated by the elements $\mathcal{I}^{\mu!},$ where $I_{\mu}$ runs through all possible subsets of $n$ elements in $\{1,\dots ,m\}$. 
\end{definition}
It is understood that for each such subset $I_{\mu}$ there is a unique strictly decreasing map $\mu: I_{\mu}\to \{1,\dots, n\}$, so we may view $\mu$ as the element $(0,\mu)$ in $X_G\times S_{n,m}$ as above. 
\par\medskip
The subspace $\Theta\otimes\Qlb\subset K(D_{I_H\times I_G}(\Pi(F)))\otimes\Qlb$ is different from the group $K(D_{I_H\times I_G}(\Pi(F)))\otimes\Qlb$. For example, each function from $\Theta\otimes\Qlb$ vanishes at $0$ in $\Pi(F)$. This submodule $\Theta$ is the geometrization of the  first term of the  Kudla's filtration on $K(D_{I_H\times I_G}(\Pi(F)))\otimes\Qlb$.
\par\medskip
Our calculation yields the following generalization: 
\begin{proposition}
\label{Cnm}
The module $\Theta$ is free  module of rank $C_{m}^{n}$ over $K(P_{I_G}(\Fl_G))$. The elements $\mathcal{I}^{\mu!},$ where  $I_{\mu}$ runs through all possible subsets of $n$ elements in $\{1,\dots ,m\},$ form a basis of this module over $K(P_{I_G}(\Fl_G))$. 
\end{proposition}
 
Our purpose now is to show that $\Theta$ is a submodule with respect to the right action of  $K(P_{I_H}(\flagvar{H}))\otimes \qelbar$ on $K(D_{I_H\times I_G}(\Pi(F)))$ and identify $\Theta$ as the induced representation from a parabolic subalgebra. The considerations are essentially on the level of Grothendieck groups. Let us simply denote  $K(D_{I_H\times I_G}(\Pi(F))\otimes\Qlb$ by $\mathcal{S}.$  Remind that $\cS_{0}$ is the $\Qlb$-subspace of $\mathcal{S}$ generated by the elements $\cI^{(w\centerdot w_0)!}$, where  $w$ runs through  $\widetilde{W}_G$. 
\par\medskip
Denote by $M$ a standard Levi subgroup of $H,$ and by $W_{M}$ the corresponding finite Weyl group.  Let $I_{M}=M(F)\cap I_{H}.$
Denote by $T_{w}$ the characteristic function of the double coset $IwI$ for any $w$ in $\weyl{H}$. The algebra $\iwahorihecke{M}$ is the subalgebra of $\iwahorihecke{H}$ generated by $(T_{w})_{w\in W_{M}},$ and by the Bernstein functions $(\theta_{\lambda})_{\lambda\in{X_{H}}}.$ Remind that our convention for the Wakimoto objects is the one in \cite{Prasad}. According to \cite[\S\ 5.4]{Prasad}, each coset $W_{M}\backslash W_{H}$ has a unique element of minimal length. Let ${}^M \! W_{H}$
be the set of such elements. If $\Delta_{M}$ denotes the simple roots of $M$ then 
$${}^M \! W_{H}=\{w\in W_{H}\vert w(\check{\alpha})>0 \hspace{2mm}\mathrm{for}\hspace{2mm} \mathrm{each}\hspace{2mm} \check{\alpha}\hspace{1mm}in\hspace{1mm}\Delta_{M}\}.$$
Any $w$ in $W_{H}$ can be written as $w^{''}w^{'},$  where $w^{''}$ and $w^{'}$  are respective elements of $W_{M}$ and ${}^M \! W_{H}$ satisfying $\ell(w)=\ell(w^{''})+\ell(w^{'}).$ Therefore $T_{w}$ equals $T_{w^{''}}T_{w^{'}}.$
 We recall that $\iwahorihecke{H}$ is a free module  over $\iwahorihecke{M}$ generated by $\{T_{w^{'}}\vert w^{'}\hspace{1mm}in \hspace{1mm}{}^M\! W_{H}\}.$ We are going to prove the two following results:
\begin{theorem}
\label{S0}
The space $\mathcal{S}_{0}$ is a submodule of $\mathcal{S}$ for the right action of $\iwahorihecke{M}.$  
\end{theorem}
The inclusion of $\mathcal{S}_{0}$ in $\mathcal{S}$ is a homomorphism of right $\mathcal{H}_{I_{M}}$-modules and left $\mathcal{H}_{I_{G}}$-modules. By adjonction, we get a morphism 
$$\alpha:\mathcal{S}_{0}\otimes_{\iwahorihecke{M}} \iwahorihecke{H} \to \mathcal{S}$$
of right $\iwahorihecke{H}$-modules and left $\iwahorihecke{G}$-modules. 
\begin{theorem}
\label{avali}
The map $\alpha:\mathcal{S}_{0}\otimes_{\iwahorihecke{M}} \iwahorihecke{H} \to \mathcal{S}$ is injective, and its image equals $\Theta\otimes \qelbar.$
\end{theorem} 
The rest of the section is devoted to the proof of these two theorems. The following lemma proves that $\mathcal{S}_{0}$ is a free module of rank one over $\iwahorihecke{G}$.
\begin{lemma}
\label{LwIw0G}
For any element $w$ in $\weyl{G},$ we have 
$$\heckefunc{G}(L_{w!},I^{w_0!})\iso \mathcal{I}^{(w\centerdot w_0)!}.$$
\end{lemma}
\begin{proof}
For a point $gI_{G}$ in $\overline{\flagvar{G}^{w}}$, let $L'=gL$ and equip $L'/tL'$ with the flag $L_i^{'}=g(L_i)$, for $1\leq i\leq n.$ 
Let $\Pi_{\mathcal{I}^{w_{0}},r}\newtimes \overline{\flagvar{G}^{w}}$ be the scheme classifying pairs $(v,gI_{G})$, where $gI_{G}$ is in $\overline{\flagvar{G}^{w}}$, and $v$ is a map  from $U^*$ to $L'/t^{r}L$ such that the induced map 
$$\overline{v}:U^*/tU^{*}\longrightarrow L'/tL'$$
sends $u^*_m,\dots,u_{n+1}^{*}$ to zero and $\textrm{Vect}(u^{*}_n,\dots,u^{*}_{n-i} )$ to $L'_{i+1}$ for $i=0,\dots,n-1.$ Let 
\begin{equation}
\label{mappi0}
\pi:\Pi_{\mathcal{I}^{w_{0}},r}\newtimes \overline{\flagvar{G}^{w}}\longrightarrow \Pi_{N,r}
\end{equation}
be the proper map  sending a couple $(v,gI_{G})$ to $v.$ By definition we have 
$\heckefunc{G}(L_w,\mcI^{w_{0}})\iso \pi_{!}(\qelbar\bt L_w).$ Let $\Pi_{\mathcal{I}^{w_{0}},r}^{0}\newtimes \flagvar{G}^{w}$ be the open subscheme of $\Pi_{\mathcal{I}^{w_{0}},r}\newtimes \overline{\flagvar{G}^{w}}$ consisting of pairs $(v,I_{G})$ such that $gI_G$ is in $\flagvar{G}^{w},$  and the map $\overline{v}:Vect(u^*_{n},\dots,u^{*}_{n-i})\longrightarrow L'_{i+1}$ is an isomorphism for $i=0,\dots,n-1.$ Then $\heckefunc{G}(L_{w!},\mathcal{I}^{w_0!})\iso \pi_{!}^{0}(\qelbar\bt L_w),$ where $\pi^{0}:\Pi_{\mathcal{I},r}^{0}\newtimes \mathcal{F}\ell_{G}^{w}\longrightarrow\Pi_{N,r}$ is the restriction of $\pi.$ The image of $\pi^{0}$ equals $\Pi_{N,r}^{w\centerdot w_0}$ and $\pi^{0}$ is an isomorphism onto its image. Thus we have $\heckefunc{G}(L_{w!},\mathcal{I}^{w_0!})\iso \mathcal{I}^{(w\centerdot w_0)!}.$
\end{proof}

Now let $w=t^{\lambda}\tau$ be an element of $\weyl{H}.$ The cocharacter $\lambda$ in $X_{H}$ is of the form $(a_{1},\dots,a_{m})$ with $a_{i}$ in $\Z$. Choose  two integers $N,r$ such that $-N\leq a_{i}< r$ for all $i$. Denote by $U_{1}\subset U_{2}\subset \dots \subset U_{m}$  the standard flag over $U/tU.$ We define the scheme $\Pi_{\mathcal{I}^{w_{0}},r}\times \overline{\flagvar{H}^{w}}$  in the same way we did for $G$. For any point  $hI_{H}$ in $\overline{\flagvar{H}^{w}}$, we put $U^{'}=hU$ and equip $U^{'}/tU^{'}$ with the complete flag $U^{'}_{i}=hU_{i}.$ Then  $\Pi_{\mathcal{I}^{w_{0}},r}\times \overline{\flagvar{H}^{w}}$ is the scheme classifying pairs $(v,hI_{H}),$ where  $hI_{H}$ is in $\overline{\flagvar{H}^{w}}$ and  $v$ is a map  from $L^{*}$ to $U^{'}/t^{r}U$ such that the induced map
$$\overline{v}:L^{*}/tL^{*}\longrightarrow U^{'}/tU^{'}$$
sends $\Vect(e_{n}^{*},\dots,e_{n-i}^{*})$ to $U_{i+1}^{'}$ for all $i=1,\dots,n-1.$ Let $\pi$ be the projection 
$$\pi: \Pi_{\mathcal{I}^{w_{0}},r}\times \overline{\flagvar{H}^{w}}\longrightarrow\Pi_{N,r}.$$
Then by definition we obtain $\heckefunc{H}(L_{w},\mcI^{w_{0}})\iso \pi_{!}(\qelbar\bt L_{w}).$ Let $\Pi_{\mathcal{I}^{w_{0}},r}^{0}\times \flagvar{H}^{w}$ be the open subscheme of $\Pi_{\mathcal{I}^{w_{0}},r}\times \overline{\flagvar{H}^{w}}$ defined by the additional condition that the above map  $\overline{v}:L^{*}/tL^{*}\longrightarrow U_{n}^{'}$ is an isomorphism. Let 
\begin{equation}
\label{pi0}
 \pi^{0}:\Pi_{\mathcal{I}^{w_{0}},r}^{0}\times \flagvar{H}^{w}\longrightarrow \Pi_{N,r}
 \end{equation}
  be the restriction of  $\pi.$  Then we have 
$$\heckefunc{H}(L_{w!},\mcI^{w_{0}!})\iso \pi^{0}_{!}(\qelbar\bt L_{w}).$$

\begin{lemma}
\label{LwIw0H}
Assume that $\lambda=(0,\dots,0,a_{n+1},\dots, a_{m})$ and that the coefficients $a_{i}$ are non negative. If $\tau$ is a permutation acting trivially on $\{1,\dots,n\}$ and permuting $\{n+1,\dots,m\},$ we have 
$$\heckefunc{H}(L_{w!},\mcI^{w_{0}!})\iso \mcI^{w_{0}!}[n\langle\lambda , \check{\omega}_{m}\rangle-\ell(w)].$$
\end{lemma}

\begin{proof}
In this case, the image of the map $\pi^{0}$ $\eqref{pi0}$ is exactly the orbit $\Pi_{0,r}^{w_0}$ and the fibre of $\pi^{0}$ is an affine space. We need to compute the dimension of the fibres. Note that $\Pi_{0,r}^{w_0}$ has dimension $(r-1)mn+\frac{n(n+1)}{2}.$ The scheme $\Pi_{\mathcal{I}^{w_{0}},r}^{0}\newtimes \mathcal{F}\ell_{H}^{w}$ is of dimension $\ell(w)-n\langle\lambda,\check{\omega}_{m}\rangle+(r-1)nm+\frac{n(n+1)}{2}.$  This implies that the dimension of the fibre of $\pi^{0}$ equals $\ell(w)-n\langle \lambda,\check{\omega}_{m}\rangle.$ This yields the result. 
\end{proof}

\begin{proposition}
\label{tauw0}
Let $\tau$ be in $W_{H}.$ Then 
$$\heckefunc{H}(L_{\tau!},\mcI^{w_{0}!})\iso \mcI^{\nu!}[r],$$
where $\nu=(0,w_{0}\tau^{-1})$ is an element of $X_{H}\times S_{n,m}$ and $r$ is the dimension of the fibre of the map $\pi^{0}$ in $\eqref{pi0}$.
\end{proposition}
Before proving this proposition  we will need the following lemma:

\begin{lemma}
\label{affinefibres}
Let $U_{1}\subset\dots \subset U_{m}=U_{0}$ be a complete flag on $U_{0}$.
Consider a partial flag 
$$V_{1}\subset V_{2}\subset \dots \subset V_{n}\subset U_{0}$$
inside $U_{0}.$
Let $\tau$ be a reflection in  the finite Weyl group $W_{H}$ of $H.$ Denote by $Y$ the variety of complete flags
$$V_{1}^{'}\subset \dots\subset V_{m}^{'}$$
which are in relative position $\tau$ with respect to the standard complete flag $U_{1}\subset\dots\subset U_{m}$
such that $V_{i}^{'}=V_{i}$ for all $i=1,\dots, n.$
Then the variety $Y$ is isomorphic to a finite dimensional affine space.
\end{lemma}

\begin{proof}
The stabilizer of the complete flag $U_{1}\subset \dots \subset U_{m}=U_{0}$ in $H$ is the Borel subgroup $B_{H}.$ For $i=1,\dots ,m$ fix a basis $u_{i}$ of $U_{0}$
 such that $U_{i}=\Vect(u_{1},\dots,u_{i}).$ The variety $H/B_{H}$ is identified with the complete flags in $U_{0}.$ Given a vector subspace $V$ of $U_{0}$ of dimension $k$, we associate to this subspace a subset $I(V)$ of $k$ elements in $\{1,\dots,m\}$ defined by  $I(V)=\{1\leq i\leq m \vert \dim(V\cap U_{i})>\dim(V\cap U_{i-1})\}.$ 
 Thus, for $w$ in $W_{H}$, the orbit $B_{H}wB_{H}/B_{H}$ is the variety of complete flags 
 $$U_{1}^{'}\subset \dots \subset U_{m}^{'}$$
 on $U_{0}$ such that $1\leq i\leq m$ we have  $I(U_{i}^{'})=\{w(1),\dots w(i)\}$.
 \par\medskip
 Let $\mathcal{V}_{n}$ be a flag $V_{1}\subset \dots \subset V_{n}\subset U_{0}$ such that $\dim(V_{i})=i.$ The space $Y$ is the variety of complete flags $V_{1}^{'}\subset \dots 
 \subset V_{m}^{'}$ lying in the orbit $B_{H}wB_{H}/B_{H}$ and satisfying $V_{i}^{'}=V_{i}$ for all $1\leq i\leq n.$ In order that the space $Y$ be non-empty, we must have 
 $I(V_{n})=\{w(1),\dots w(n)\}.$ Assume that this is true. 
 Given a subset of $k$ elements $I_{k}$ in $\{1,\dots m\}$, denote by $Z_{I_{k}}$ the variety of subspaces $V$ of $U_{0}$ such that $I(V)=I_{k}$ (in particular we have $\dim( V)=k$).
 Given another subset $I_{k+1}$ of $k+1$ elements containing  $I_{k}$, let $Z_{I_{k},I_{k+1}}$ be the variety of pairs $(V\subset V^{'})$, where $V$  lies in $Z_{I_{k}}$ and $V^{'}$ lies in $Z_{I_{k+1}}.$ Denote by $\pi$ the projection from $Z_{I_{k},I_{k+1}}$ onto $Z_{I_{k}}$ sending $(V\subset V^{'})$ to $V$. Let us prove that the map $\pi$ is $B_{H}$-equivariant affine fibration. 
 \par\smallskip
 For $V$ in $Z_{I_{k}}$ denote by $\overline{U}_{i}$ the image of $U_{i}$ under the map $U_{0}\to U_{0}/V.$ Then $\overline{U}_{i}=\overline{U}_{i-1}$ if and only if $i$ lies in $I_{k}.$ Denote by  $s$ the single element of $I_{k+1}-I_{k}$. The fibre of the map $\pi$ identifies with the variety of $1$-dimensional subspaces $V^{'}/V \subset U_{0}/V$ such that $V^{'}/V$ is a subset of $\overline{U}_{s}$ and $V^{'}/V$ is not contained in $\overline{U}_{s-1}.$ This fibre is affine and since the space $Z_{I_{k}}$ is $B_{H}$-homogeneous, the map $\pi$ is a $B_{H}$-equivariant affine fibration. 
 \par\medskip
 For $r\geq n$ denote by $Y_{r}$ the variety of flags $V_{1}\subset \dots \subset V_{n}\subset V_{n+1}^{'}\subset \dots \subset V_{r}^{'}$ such that $I(V_{i})=\{\tau(1),\dots,\tau(i)\}$ for $n\leq i\leq r$. We have the forgetful maps 
 $$Y_{m}\overset{f_{m}}\longrightarrow Y_{m-1}\overset{f_{m-1}}\longrightarrow \dots \overset{f_{n+1}}\longrightarrow Y_{n}=\mathrm{Spec}(\K).$$
 Any of the map $f_{i}$ above is obtained by a base change from the map $\pi$ for a suitable pair ($I_{k}\subset I_{k+1}$). The fibre of a map $f_{r}$ depends only on $V^{'}_{r-1}$ and not on the smaller 
 $V_{j}^{'}$ for $j\leq r-2$. As any affine fibration over an affine space is trivial this leads to the result and $Y$ is an affine space of dimension $r$ for some $r\geq n.$
\end{proof}

\begin{proof}[Proof of Proposition $\ref{tauw0}$]
Let  us precise the definition of $\nu$: $\nu=(0,w_{0}\tau^{-1})$ is an element of $X_{H}\times S_{n,m},$ where the set $I_{w_{0}\tau^{-1}}$ is the set $\tau(\{1,\dots,n\})$ and $w_{0}\tau^{-1}:I_{w_{0}\tau^{-1}}\to \{1,\dots,n\}$ is the corresponding bijection. Consider the map
$$\pi^{0}:\Pi_{\mcI^{w_{0}},1}^{0}\times \flagvar{H}^{\tau}\longrightarrow \Pi_{0,1}.$$  
The fibres of the map $\pi^{0}$ are affine spaces according to Lemma $\ref{affinefibres}.$ We denote by $r$ their dimension.  Since  the image of $\pi^{0}$ is  the orbit $\Pi_{0,1}^{\nu}.$ We obtain that $\heckefunc{H}(L_{\tau!},\mcI^{w_{0}!})\iso \mcI^{\nu!}[r].$
\end{proof}
\begin{remark}
If the permutation $\tau$ is a actually a permutation of $\{1,\dots,n\}$ and acts trivially on $\{n+1,\dots,m\}$ then the shift in the above formula disappears and the map $\pi^{0}$ will be an isomorphism onto its image $\Pi_{0,1}^{\nu}.$
\end{remark}

%According to Lemma $\ref{LwIw0G}$,  the subspace $\mathcal{S}_{0}$ is a free module of rank one over $\iwahorihecke{G}.$
Let $M$ be the standard Levi subgroup in $H$ corresponding to the partition $(n,m-n)$ of $m$. Then $M$ is of the form $M_{1}\times M_{2}$, where $M_{1}\iso \GL_{n}$ and $M_{2}\iso \GL_{m-n}.$ Write $\iwahorihecke{M}$ for the Iwahori Hecke algebra associated to $M$ viewed as subalgebra of $\iwahorihecke{H}$. We have naturally $\iwahorihecke{M}\iso \iwahorihecke{M_{1}}\otimes_{\qelbar}\iwahorihecke{M_{2}}.$ We will denote by $X_{M_{i}}$ the coweight lattice of $M_{i},$ for $i=1,2.$  
The space $\mathcal{S}_{0}$ is not a $\iwahorihecke{M}$-submodule for the natural left action of $\iwahorihecke{M}$ on $\mathcal{S}(\Pi(F))^{I_{H}\times I_{G}}.$ For instance, if $\lambda=(1,\dots,1,0\dots,0)$ where $1$ appears $n$ times, then the complex $\heckefunc{H}(L_{t^{\lambda!}},\mcI^{w_{0}!})$ doesn't occur in $\mathcal{S}_{0}.$  
 We will consider this right action and will show that $S_{0}$ is a right $\iwahorihecke{M}$-module  under this right action. Remind that the  right action of $\iwahorihecke{M}$ commutes with the left action of $\iwahorihecke{G}.$
  \begin{lemma}
  For $\tau$ a simple reflection in the finite Weyl group $W_{M}$ of $M$, $\heckefuncr{H}(L_{\tau!},\mcI^{w_{0}!})$ lies in $\mathcal{S}_{0}.$ 
  \end{lemma}
  \begin{proof}
According to Lemma $\ref{tauw0},$ we have
\begin{equation}
\label{simplereflection}
\heckefuncr{H}(L_{\tau!},\mcI^{w_{0}!})\iso \heckefunc{H}(L_{\tau ! },\mcI^{w_{0}!})\iso \mcI^{\nu !}[r],
\end{equation}
where $\nu=(0,w_{0}\tau)$ is viewed as an element of $X_{H}\times S_{n,m}.$   Thus $\heckefuncr{H}(L_{\tau!},\mcI^{w_{0}!})$ occurs in $\mathcal{S}_{0}.$ 
\end{proof}

\begin{lemma}
\label{lemme}
Let $\omega=(1,\dots,1)$ be in $X_{H},$ $\mu_{1}=(a_{1},\dots,a_{n})$ be in $X_{M_{1}}$ and $\mu_{2}=(a_{n+1},\dots,a_{m})$ be in $X_{M_{2}}.$  Let $\lambda$ be the coweight  $\mu_{1}+\mu_{2}$ in $X_{H}$ and assume that if $m\geq i > n\geq j\geq 1$ then $a_{i}\geq a_{j}.$ We also fix two integers $r,N$  such that $-N\leq a_{i}< r$ for all $i.$ Let $v$ be a $\locring$-linear map from $L^{*}$ to $t^{\lambda}U/t^{r}U$  such that
for $0\leq i < n,$ the induced map 
$$\overline{v}:L^{*}/tL^{*}\longrightarrow t^{\lambda}U/t^{\lambda+\omega}U$$
 sends $\Vect(e_{n}^{*},\dots ,e_{n-i}^{*})$ isomorphically onto $\Vect(t^{a_{1}}u_{1},\dots t^{a_{i+1}}u_{i+1}).$  Denote by $\nu$ the element $(w_0(\mu_{1}),w_0)$ in $X_{G}\times S_{n,m}.$ Then  $v$ is an element of the orbit $\Pi_{N,r}^{\nu}.$
\end{lemma}

 \begin{proof}
 Let  $U_{1}=\locring u_{1}\oplus\dots \oplus \locring u_{n}$  and $U_{2}=\locring u_{n+1}\oplus \dots\oplus \locring u_{m}.$ Then the map $v$ can be written as a pair $(v_{1},v_{2}),$ where $v_{i}:L^{*}\longrightarrow t^{\mu_{i}}U_{i}/t^{r}U_{i}$ for $i=1,2.$ Let $N_{G}\subset B_{G}$ be unipotent radical of the standard Borel subgroup of $G$. Acting by a suitable element of $N_{G}\subset I_{G}$ on $v$, one may assume that $v_{1}(e^{*}_{i})=t^{a_{w_{0}(i)}}u_{w_{0}(i)}$ modulo $t^{\mu_{1}+\omega}U_{1}.$
\par\medskip
Furthermore consider the groups 
$$I_{1}=\{g\in{\GL(U_{1})\hspace{1mm}\vert \hspace{1mm}g=\mathrm{id} \hspace{1mm}\mathrm{mod}\hspace{1mm} t}\}$$
 and 
 $$I_{G,0}=\{g\in \GL(L)\hspace{1mm}\vert\hspace{1mm} g=\mathrm{id} \hspace{1mm}\mathrm{mod}\hspace{1mm} t\}.$$ 
 
 Acting by  a suitable element of $I_{G,0}\times I_{1}$ we may assume that $v_{1}(e^{*}_{i})=t^{a_{w_0 (i)}}u_{w_0( i)}.$ This implies that $\overline{v}_{2}$ vanishes. 
 Now viewing $v$ as a map from $L^{*}$ to $t^{-N}L,$ we observe that $r$ can be replaced by $1+\mathrm{min}\{{a_{n+1},\dots,a_{m}}\}.$ Hence $v$ is an element of $\Pi_{N,r}^{\nu}.$
 \end{proof}

\begin{lemma}
\label{M1action}
Let $\lambda=(a_{1},\dots,a_{n},0,\dots,0)$  be  a anti-dominant cocharacter in $X_{H},$  in particular all  $a_{i}$'s are nonpositive. Then we have a canonical isomorphism 
$$\overset{\leftarrow}{H}_{H}(L_{t^{\lambda !}},\mathcal{I}^{w_{0} !})\iso \mathcal{I}^{\mu !}[\langle\lambda,2\check{\rho}_{G}-2\check{\rho}_{H}\rangle+(n-m)\langle\lambda,\check{\omega}_{m}\rangle],$$
where $\mu=(w_0( \lambda),w_0)$ in $X_{G}\times S_{n,m}$ and $w_0$ is the longest element of the finite Weyl group of $M_1$. We identify $M_1$ with $G$.
\end{lemma}

\begin{proof}
Consider the map  
$$\pi^{0}:\Pi_{\mathcal{I}^{w_{0}},r}^{0}\newtimes \mathcal{F}\ell_{H}^{w}\longrightarrow \Pi_{N,r}$$ 
defined in $\eqref{mappi0}.$ By applying  Lemma $\ref{lemme}$ to $w=t^{\lambda},$ we see that the image of $\pi^{0}$ is the $I_{H}\times I_{G}$-orbit on $\Pi_{N,r}$ passing through the map $v$ from $L^{*}$ to  $t^{-N}U/t^{r}U$ given by $v(e_{i}^{*})=t^{a_{w_{0}(i)}}u_{w_{0}(i)}$ for $1\leq i \leq n.$ This orbit corresponds to element $\mu=(w_{0}(\lambda), w_0)$ in $X_{G}\times S_{n,m}.$ 
Restricting $\pi^{0}$ to its image we get a morphism
\begin{equation}
\label{newactionH}
\pi^{0}_{H}:\Pi_{\mathcal{I}^{w_{0}},r}^{0}\newtimes \mathcal{F}\ell_{H}^{t^{\lambda}}\longrightarrow \Pi_{N,r}^{\mu}
\end{equation}  

whose fibres are affine spaces. One has $\dim( \mathcal{F}\ell_{H}^{t^{\lambda}})=\langle \lambda,2\check{\rho}_{H}\rangle.$ For any point $hI_{H}$  in $\mathcal{F}\ell_{H}^{t^{\lambda}},$  let $U^{'}=hU.$ Then  
$$ \dim (\mathrm{Hom}_{\locring}(L^{*},t^{\lambda+\omega}U/t^{r}U))=nm(r-1)-n\langle\lambda,\check{\omega}_{m}\rangle.$$
Thus the affine space of maps from $L^{*}$ to $hU/t^{r}U$ sending $\Vect(e^{*}_{n},\dots,e_{n-i}^{*})$ to $U_{i+1}^{'}$ for $i=1,\dots ,n-1$ is of dimension $\frac{n^{2}+n}{2}+nm(r-1)-n\langle \lambda,\check{\omega}_{m}\rangle.$ Finally 
$$\dim(\Pi_{\mcI^{w_{0}},r}^{0}\newtimes \flagvar{H}^{t^{\lambda}})=\langle\lambda, 2\check{\rho}_{H} \rangle+\frac{n^{2}+n}{2}+nm(r-1)-n\langle\lambda ,\check{\omega}_{m}\rangle.$$
Moreover, we have the following isomorphism 
$$\pi^{0}_{G}:\Pi_{\mathcal{I}^{w_{0}},r}^{0}\newtimes \flagvar{G}^{t^{w_{0}(\lambda})}\iso \Pi_{N,r}^{\mu}.$$

By using  $\dim (\mathcal{F}\ell_{G}^{t^{w_{0}\lambda}})=\langle\lambda,2\check{\rho}_{G}\rangle,$ we get that the dimension of $\Pi_{N,r}^{\mu}$ equals 
$$\frac{n^{2}+n}{2}+nm(r-1)-m\langle\lambda,\check{\omega}_{m}\rangle+\langle\lambda,2\check{\rho}_{G}\rangle$$
and hence the dimension of the fibres of the map $\eqref {newactionH}$ equals
 $$\langle \lambda,2(\check{\rho}_{G}-\check{\rho}_{H})\rangle+(m-n)\langle\lambda,\check{\omega}_{m}\rangle$$ 
 which allows us to calculate the announced shift in the Lemma. Additionally this proves that $\overset{\leftarrow}{H}_{H}(L_{t^{\lambda !}},\mathcal{I}^{w_{0} !})$ lies in $\mathcal{S}_{0}.$
\end{proof}
Remind the following two properties  due to \cite{Arkhipov},
 \begin{enumerate}
 \label{equaarkhi}
 \item
 If $w_{1},w_{2}\in{\weyl{G}}$ verify $\ell(w_{1}w_{2})=\ell(w_{1})+\ell(w_{2})$ then we have a canonical isomorphism
\begin{equation}
\label{trio}
L_{w_{1}*}\star L_{w_{2}*}\iso L_{w_{1}w_{2}*}.
\end{equation}
Under the same assumption, and by duality the same result is true for $L_{w!}.$
\item Denote by $e$ the identity element of $\weyl{G}$ then for any $w\in{\weyl{G}},$ we have 
\begin{equation}
\label{e}
L_{w!}\star L_{w^ {-1}*}\iso L_{w ^ {-1}*}\star L_{w!}\iso L_{e}.
\end{equation}  
Hence the perverse sheaf $L_{w!}$ is an invertible object of $D_{I_{G}}(\flagvar{G}).$
 \end{enumerate}
\begin{proposition}
\label{13}
Let $\lambda=(a_{1},\dots,a_{n},0,\dots,0)$ be an anti dominant cocharacter in $X_{H},$ in particular all $a_{i}$'s are non-positive.  Then 
$\heckefunc{H}(L_{t^{-\lambda *}},\mathcal{I}^{w_{0} !})$ occurs in $\mathcal{S}_{0}.$
\end{proposition}

\begin{proof}
According to Lemma $\ref{M1action},$ we have 
$$\heckefunc{H}(L_{t^{\lambda !}},\mcI^{w_{0}})\iso \mcI^{\mu!}[d],$$
where   the shift $d$ equals $[\langle\lambda,2\check{\rho}_{G}-2\check{\rho}_{H}\rangle+(n-m)\langle\lambda,\check{\omega}_{m}\rangle]$. Moreover according to $\eqref{e}$, $L_{t^{-\lambda}*}\star L_{t^{\lambda}!}$ is isomorphic to $L_{e}$ where $e$ is the identity element in the finite Weyl group  of $M_{1}$. Combining these two isomorphisms we obtain 
\begin{align}
\label{equa11}
\mcI^{w_{0}!} \iso & \heckefunc{H}(L_{t^{-\lambda*}}\star L_{t^{\lambda!}},\mcI^{w_{0}!}) \iso  \heckefunc{H}(L_{t^{-\lambda*}},\mathcal{I}^{\mu!})[d]\nonumber\\
\iso & \heckefunc{H}(L_{t^{-\lambda *}},\heckefunc{G}(L_{t^{w_{0}(\lambda )!}},\mathcal{I}^{w_{0}!}))[d]\nonumber\\
 \iso & \heckefunc{G}(L_{t^{w_{0}(\lambda )!}},\heckefunc{H}(L_{t^{-\lambda*}},\mcI^{w_{0}}))[d], 
\end{align}
where the third isomorphism is due to Lemma $\ref{LwIw0G}$ and the last one is due the fact that the actions of $H$ and $G$ commute. 
Applying $\heckefunc{G}(L_{t^{-w_{0}(\lambda )*}},.)$ to both sides of $\eqref{equa11}$, we obtain
\begin{equation}
\label{boos}
\heckefunc{G}(L_{t^{-w_{0}(\lambda)*}},\mathcal{I}^{w_{0}!})\iso \heckefunc{H}(L_{t^{-\lambda*}},\mathcal{I}^{w_{0}})[d].
\end{equation}
Since $\mathcal{S}_{0}$ is a left $\iwahorihecke{G}$-module, the left hand side of $\eqref{boos}$ lies in $\mathcal{S}_{0}$. Thus so does the right hand side.  
\end{proof}

\begin{lemma}
\label{10}
Let $\lambda=(0,\dots,0,a_{n+1},\dots,a_{m})$ be a dominant cocharacter of $M.$ If $a_{n+1}\geq \dots \geq a_{m}\geq 0$ then
$$\heckefunc{H}(L_{t^{-\lambda}*},\mcI^{w_{0}!})\iso \mcI^{w_{0}!}[\langle\lambda, 2\check{\rho}_{H} \rangle-n\langle \lambda, \check{\omega}_{m}\rangle].$$
\end{lemma}

\begin{proof}
 Lemma $\ref{LwIw0H}$ applied to $w=t^{\lambda}$ ($\tau$ being the identity) gives us 
$$\heckefunc{H}(L_{t^{\lambda!}},\mcI^{w_{0}})\iso \mcI^{w_{0}}[n\langle \lambda, \check{\omega}_{m}\rangle-\langle \lambda 2\check{\rho}_{H}\rangle].$$ This implies the assertion for $L_{t^{-\lambda}*}.$
\end{proof}

\begin{proposition}
\label{shift}
Let $\lambda$ be a dominant cocharacter in $X_{H}$  that can be written as the sum of two cocharacters $\lambda_{1}$ and $\lambda_{2}$ in $X_{M_{1}}$ and $X_{M_{2}}$ respectively.  If $\nu=(-w_{0}(\lambda_{1}),w_{0})$ then 
$$\heckefuncr{H}(L_{t^{\lambda !}},\mathcal{I}^{w_{0}!})\iso \heckefunc{H}(L_{t^{-\lambda !}},\mathcal{I}^{w_{0} !})\iso \mathcal{I}^{\nu !}[\langle\lambda_{1},2\check{\rho}_{G}\rangle-\langle\lambda,2\check{\rho}_{H}\rangle+\langle m\lambda_{1}-n\lambda,\check{\omega}_{m}\rangle],$$
where we identify $M_{1}$ with $G$ and hence  $\check{\rho}_{M_{1}}$ with $\check{\rho}_{G}.$ Thus 
$\heckefuncr{H}(L_{t^{\lambda !}},\mathcal{I}^{w_{0}!})$ occurs in $\mathcal{S}_{0}.$
\end{proposition}

\begin{proof}
Set $- \lambda=(a_{1},\dots ,a_{m})$ and choose two integers $r,N$ such that $-N\leq a_{i}< r
$ for all $i.$ By Lemma $\ref{lemme}$ the map 
$$\pi^{0}:\Pi_{\mathcal{I}^{w_{0}},r}^{0}\newtimes \mathcal{F}\ell_{H}^{t^{-\lambda}}\longrightarrow \Pi_{N,r}$$
 factors through $\Pi_{N,r}^{\nu}$ by a map $\pi^{0}_{H},$ where $\nu$ equals $(-w_{0}(\lambda_{1}),w_{0}).$ The dimension of $\Pi_{\mathcal{I}^{w_{0}},r}^{0}\newtimes \mathcal{F}\ell_{H}^{t^{-\lambda}}$ equals $nm(r-1)+n\langle\lambda,\check{\omega}_{m}\rangle+\frac{n^{2}+n}{2}+\langle\lambda,2\check{\rho}_{H}\rangle.$ We have the isomorphism 
 
 $$\Pi_{\mathcal{I}^{w_{0}},r}^{0}\newtimes \mathcal{F}\ell_{H}^{t^{-w_{0}(\lambda_{1})}}\iso \Pi_{N,r}^{\nu}$$

and this allows us to calculate the dimension of $I_{H}\times I_{G}$-orbit $\Pi_{N,r}^{\nu}.$ Namely 
$$\dim (\mathcal{F}\ell_{G}^{t^{-w_{0}(\lambda_{1})}})=\langle -w_{0}(\lambda_{1}),2\check{\rho}_{G}\rangle=\langle\lambda_{1},2\check{\rho}_{G}\rangle.$$ 
Remind that $w_{0}$ is longest element of the finite Weyl group of $M_{1}\iso G.$ This yields

$$\dim(\Pi_{n,r}^{\nu})=nm(r-1)+m\langle\lambda_{1},\check{\omega}_{m}\rangle+\frac{n^{2}+n}{2}+\langle\lambda_{1},2\check{\rho}_{G}\rangle.$$ 

So the dimension of a fibre of the map  $\pi^{0}_{H}$ is $\langle \lambda,2\check{\rho}_{H}\rangle-\langle\lambda_{1},2\check{\rho}_{G}\rangle+\langle n\lambda-m\lambda_{1},\check{\omega}_{m}\rangle.$ This justifies the shift in the formula announced above and the assertion follows.  
\end{proof}
\begin{remark}
\label{remark2}
 In Lemma $\ref{shift}$ if  $\lambda_{2}$ equals $0$  then the corresponding map $\pi^{0}_{H}$ is an isomorphism and the shift in the above formula disappears. 
\end{remark}

\begin{proposition}
\label{15}
For any $\lambda$ in $X_{H}^{+}$, the complex $\heckefunc{H}(L_{t^{\lambda}*},\mathcal{I}^{w_{0}})[d]$ occurs in $\mathcal{S}_{0}.$  Besides   for a dominant cocharacter $\mu$ in $X_{H}^{+}$, the complex $\heckefunc{H}(L_{t^{-\mu}!}\star L_{t^{\lambda}*},\mathcal{I}^{w_{0}})$ occurs in $\mathcal{S}_{0}$ as well.  
\end{proposition}

\begin{proof}
The equality $\eqref{e}$ combining this with Proposition $\ref{shift},$ we get :
\begin{align}
\mathcal{I}^{w_{0}!} & \iso \heckefunc{H}(L_{t^{\lambda}*}\star L_{t^{-\lambda}!},\mathcal{I}^{w_{0}!})\iso \heckefunc{H}(L_{t^{\lambda}*},\mathcal{I}^{\nu})[d] \nonumber\\
& \iso \heckefunc{H}(L_{t^{\lambda}*},\heckefunc{G}(L_{t^{-w_{0}(\lambda)!}},\mathcal{I}^{w_{0}!}))[d] \nonumber\\
& \iso \heckefunc{G}(L_{t^{-w_{0}\lambda !}},\heckefunc{H}(L_{t^{\lambda}*},\mathcal{I}^{w_{0}!}))[d],.
\end{align}

The shift $d$ is also the one defined in Proposition $\ref{shift}$. The third isomorphism is  due to  Lemma $\ref{LwIw0G}$ and the fourth holds by using the commutativity of the action of $G$ and $H$. 
 Applying $\heckefunc{G}(L_{t^{w_{0}(\lambda})*},.)$ to both sides, we get
$$\heckefunc{G}(L_{t^{w_{0}(\lambda)*}},\mathcal{I}^{w_{0}!})\iso \overset{\leftarrow}{H}_{H}(L_{t^{\lambda}*},\mathcal{I}^{w_{0}!})[d].$$ 
This complex occurs in $\mathcal{S}_{0}.$
Now consider  the isomorphism 
\begin{align}
\heckefunc{H}(L_{t^{-\mu}!}\star L_{t^{\lambda }*},\mathcal{I}^{w_{0}!}) & \iso \overset{\leftarrow}{H}_{H}(L_{t^{-\mu}!},  \heckefunc{G}(L_{t^{w_{0}(\lambda})*},\mathcal{I}^{w_{0}}))[-d]\nonumber \\
& \iso \heckefunc{G} (L_{t^{w_{0}(\lambda)}*},\overset{\leftarrow}{H}_{H}(L_{t^{-\mu }!},\mathcal{I}^{w_{0}}))[-d].
\end{align}
By Lemma $\ref{shift}$, the complex $\heckefunc{H}(L_{t^{-\mu!}},\mcI^{w_{0}!})$ occurs in $\mathcal{S}_{0}$. Since $\mathcal{S}_{0}$ is a $\iwahorihecke{G}$-module we can  apply the functor $\heckefunc{G}(L_{t^{w_{0}(\lambda)}*},)$ to $\heckefunc{H}(L_{t^{-\mu!}},\mcI^{w_{0}!}).$ Then the result will still occur in $\mathcal{S}_{0}.$
\end{proof}

\textbf{proof of Theorem} $\mathbf{\ref{S0}:}$
The assertion follows from $\eqref{simplereflection}$, Lemmas $\ref{M1action}$ and $\ref{10}$, Propositions $\ref{shift},$ $\ref{13}$ and $\ref{15}.$

\textbf{Proof of Theorem} $\mathbf{\ref{avali}:}$
Lemma $\ref{tauw0}$ and Proposition $\ref{Cnm}$ imply that the image of the map $\alpha$ is exactly $\Theta\otimes\qelbar$.
More precisely, if  $\tau$ runs through ${}^M \! W_{H}$ the elements $L_{\tau!}$ form a basis of the left $\iwahorihecke{M}$-module $\iwahorihecke{H}$. Hence for $w$  and $\tau$ runs through  $\weyl{G}$  and ${}^M \! W_{H}$ respectively, the objects 
\begin{equation}
\label{toto}
\heckefuncr{H}(L_{\tau!},\heckefunc{G}(L_{w!},\mathcal{I}^{w_{0!}}))
\end{equation}
form a basis of $\mathcal{S}_{0}\otimes_{\iwahorihecke{M}}\iwahorihecke{H}$ over $\qelbar.$
An element $\nu$ in $W_{H}$ lies in ${}^M \! W_{H}$ if and only if $\nu$ is strictly increasing on $\{1,\dots,n\}$ and on $\{n+1,\dots,m\}.$ 
For $\tau$ in ${}^M \! W_{H}$ let $\mu=w_{0}\tau$  and $I_{\mu}=\tau^{-1}(\{1,\dots,n\}).$ Consider $\mu$ as a map  from $I_{\mu}$ to $\{1,\dots,n\}$ and so as an element of $X_{G}\times S_{n,m}.$ The map $\tau^{-1}w_{0}:\{1,\dots,n\}\to I_{\mu}$ is strictly decreasing because $\tau^{-1}:\{1,\dots,n\}\to I_{\mu}$ is strictly increasing. 
 According to Lemma $\ref{tauw0}$ we have 
$$\heckefuncr{H}(L_{\tau!},\mcI^{w_{0}!})\iso \mcI^{\mu!}[r].$$
By Proposition $\ref{grothen4}$ we have $\heckefunc{G}(L_{w!},\mcI^{\mu})\iso \mcI^{w\centerdot \mu}[d^{'}]$ for some $d^{'}$ and hence the image of 
$\eqref{toto}$ under the map $\alpha$ is $\mcI^{w\centerdot \mu !}[d^{''}]$ for some shift $d^{''}$. 
\par\bigskip

The Iwahori-Hecke algebra $\iwahorihecke{M}$ identifies canonically with $\iwahorihecke{M_{1}}\otimes_{\qelbar}\iwahorihecke{M_{2}}.$ The right action of $\iwahorihecke{M_{1}}$ and $\iwahorihecke{M_{2}}$ on $\mathcal{S}_{0}$ commute with each other. We are now going to define the action of the Wakimoto sheaves on $\mcI^{w_{0}!}$.

\begin{lemma}
 We have the following isomorphisms: 
\begin{enumerate}
\item   For  any $\lambda$ in $X_{M_{2}}$ 
$$\heckefuncr{H}(L_{t^{\lambda!}},\mcI^{w_{0}!})\iso \heckefunc{H}(L_{t^{-\lambda!}},\mcI^{w_{0}!})\iso \mcI^{w_{0}!}[-\langle \lambda,2\check{\rho}_{M_{2}} \rangle].$$
\item For $\lambda$ in $X_{M_{2}}$
\begin{equation}
\label{t*}
 \heckefuncr{H}(L_{t^{-\lambda*}},\mcI^{w_{0}!})\iso \mcI^{w_{0}!}[\langle\lambda,2\check{\rho}_{M_{2}}\rangle].
 \end{equation}
\item For any $\lambda$ in $X_{M_{2}},$
$$\heckefuncr{H}(\Theta_{\lambda},\mcI^{w_{0}!})\iso \mcI^{w_{0}!}[-\langle \lambda, 2\check{\rho}_{M_{2}}\rangle],$$
where $\Theta_{\lambda}$ is the Wakimoto sheaf associated to $\lambda$.
\item For $w$ in $\weyl{G}$ and $\lambda$ in $X_{M_{2}}$, 
$$\heckefuncr{H}(\Theta_{\lambda},\mcI^{(w\centerdot w_{0})!})\iso \mcI^{(w\centerdot w_{0})!}[-\langle\lambda, \check{\rho}_{M_2} \rangle ].$$
\end{enumerate}
\end{lemma}

\begin{proof}
The first  formula is obtained by applying  Lemma  $\ref{shift}$  to the case where $\lambda_{1}=0$. The second  one is obtained from the first  and from the isomorphism $\eqref{e}$.  The third one holds by definition of $\Theta_{\lambda}$ using the two first isomorphisms.  Finally we prove the fourth one: 
\begin{align}
\heckefuncr{H}(\Theta_{\lambda}, \mcI^{(w\centerdot w_{0})!}) &\iso  \heckefuncr{H}(\Theta_{\lambda},\heckefunc{G}(L_{w!},\mcI^{w_{0}!}))\nonumber
\iso  \heckefunc{G}(L_{w!},\heckefuncr{H}(\Theta_{\lambda},\mcI^{w_{0}!})) \nonumber\\
&\iso  \heckefunc{G}(L_{w!},\mcI^{w_{0}}[-\langle\lambda, 2\check{\rho}_{M_{2}}\rangle])\iso \mcI^{(w\centerdot w_{0})!}[-\langle\lambda, 2\check{\rho}_{M_{2}}\rangle].
\end{align}
\end{proof}

\begin{corollary}
For any object $K$ in $\mathcal{S}_{0}$ and any $\lambda$ in $X_{M_{2}}$ we have 
$$\heckefuncr{H}(\Theta_{\lambda},K)\iso K [-\langle\lambda, 2\check{\rho}_{M_{2}}\rangle].$$
\end{corollary}

\begin{proposition}
\label{M2}
For any $w$ in the finite Weyl group of $M_{2}$ we have 
$$\heckefuncr{H}(L_{w!},\mcI^{w_{0}!})\iso \heckefunc{H}(L_{w^{-1}},\mcI^{w_{0}!})\iso \mcI^{w_{0}!}[-\ell(w)].$$
Thus, at the level of the functions,  the Iwahori-Hecke algebra $\iwahorihecke{M_{2}}$ acts on $\mathcal{S}_{0}$ by the character corresponding to the trivial representation of $M_{2}(F)\simeq\GL_{m-n}(F).$ Moreover, 
$$\heckefuncr{H}(\Theta_{\lambda}\star L_{\tau},K)\iso K[-\langle\lambda, 2\check{\rho}_{M_{2}}\rangle-\ell(\tau)].$$
\end{proposition}

\begin{proof}
At the level of functions,  for any $w$ in $\weyl{G}$ the character of $\iwahorihecke{G}$ corresponding to the trivial representation sends $T_{w}$, the characteristic function of the double coset $I_{G}wI_{G},$ to $q^{\ell(w)}.$ In our geometric setting,  for any $w$ in $\weyl{G}$ this character becomes the functor $\heckefunc{G}(L_{w!},)$ sending $K$ in $\mathcal{S}_{0}.$ to $K[-\ell(w)]$. Remind that the object $L_{w!}$ corresponds to $q^{-\ell(w)/2}T_{w}.$ 
\end{proof}
 Remind that $M_{1}$ is identified with $G.$ Now let us analyse the structure of $\mathcal{S}_{0}$ as a right $\iwahorihecke{M_{1}}$-module  and its relation with the left $\iwahorihecke{G}$-module structure. 
 
 \begin{lemma}
 \label{wakiM1}
$ $
\begin{enumerate}
\item For any $\tau$  in the finite Weyl group of $M_{1},$  
$$\heckefuncr{H}(L_{\tau!},\mcI^{w_{0}})\iso \mcI^{w_{0}\tau!}.$$
\item For any $\lambda$  in $X_{M_{1}}^{+},$  
$$\heckefuncr{H}(L_{t^{\lambda}!},\mcI^{w_{0}})\iso \mcI^{w_{0}t^{-\lambda}!}\, \,  \mathrm{and}\,\, 
\heckefunc{G}(L_{t^{w_{0}(\lambda)}*},\mcI^{w_{0}!})\iso \heckefuncr{H}(L_{t^{-\lambda}*}, \mcI^{w_{0}!}).$$
\item For any $\lambda$   in $X_{M_{1}}$,
$$\heckefuncr{H}(\Theta_{\lambda},\mcI^{w_{0}!})\iso \heckefunc{G}(\Theta_{-w_{0}(\lambda)},\mcI^{w_{0!}}).$$
\item For any $\tau$ in the finite Weyl group of $M_{1}$  and  any $\lambda$  in $X_{M_{1}},$ 
$$\heckefuncr{H}(\Theta_{\lambda}\star L_{\tau!},\mcI^{w_{0}})\iso \heckefunc{G}(\Theta_{-w_{0}\lambda}\star L_{w_{0}\tau w_{0}!},\mcI^{w_{0}!}).$$
\end{enumerate}
\end{lemma}

\begin{proof}
The first isomorphism is obtained from  Proposition $\ref{tauw0}.$ The second one  is a consequence of $\eqref{equaarkhi}$. For the third one, choose $\lambda_{1}$ and $\lambda_{2}$ on $X_{H}^{+}\cap X_{M_{1}}$ such that $\lambda=\lambda_{1}-\lambda_{2}.$ By definition of Wakimoto sheaves,  $\Theta_{\lambda}=L_{t^{-\lambda_{2}}*}\star L_{t^{\lambda_{1}!}}$ in $\iwahorihecke{M_{1}}.$ The isomorphism in $2)$ yields that 
\begin{align}
\heckefuncr{H}(L_{t^{\lambda_{1}!}},\heckefuncr{H}(L_{t^{-\lambda_{2}*}},\mcI^{w_{0}!})) \iso &  \heckefuncr{H}(L_{t^{\lambda_{1}!}},\heckefunc{G}(L_{t^{w_{0}(\lambda_{2})*}},\mcI^{w_{0}!}))\nonumber\\
\iso & \heckefunc{G}(L_{t^{w_{0}(\lambda_{2})*}},\heckefuncr{H}(L_{t^{\lambda_{1}!}},\mcI^{w_{0}!})) \nonumber\\
 \iso  & \heckefunc{G}(L_{t^{w_{0}(\lambda_{2})!}},\heckefunc{G}(L_{t^{-w_{0}(\lambda_{2})!}},\mcI^{w_{0}!}).
\end{align}
The element $-w_{0}\lambda_{2}$ is dominant if $\lambda_{2}$ is dominant. This implies the third assertion.
The fourth isomorphism is obtained formally in the following way:
\begin{align}
\heckefuncr{H}(\Theta_{\lambda}\star L_{\tau!},\mcI^{w_{0}!})\iso & \heckefuncr{H}(L_{\tau!},\heckefunc{G}(\Theta_{-w_{0}(\lambda)!},\mcI^{w_{0}!}))\nonumber\\ 
\iso&  \heckefunc{G}(\Theta_{-w_{0}(\lambda)},\heckefuncr{H}(L_{\tau!},\mcI^{w_{0}!}))\nonumber\\
 \iso & \heckefunc{G}(\Theta_{-w_{0}\lambda},\heckefunc{G}(L_{w_{0}\tau w_{0}!},\mcI^{w_{0}!}))\nonumber \\
 \iso & \heckefunc{G}(\Theta_{-w_{0}\lambda}\star L_{w_{0}\tau w_{0}!},\mcI^{w_{0}!}).
\end{align} 
\end{proof}

\begin{corollary}
\label{M1}
The subspace $\mathcal{S}_{0}$ is a free right $\iwahorihecke{M_{1}}$-module of rank one generated by $\mcI^{w_{0}!}.$
\end{corollary}

\begin{proof}
The assertion follows from Lemma $\ref{wakiM1}$ and the fact that if  $\lambda$  and $\tau$ runs through $X_{M_{1}}$ and $W_{M_{1}}$ respectively the elements  $\Theta_{\lambda}\star L_{\tau !}$ form a basis of $\iwahorihecke{M_{1}}.$
\end{proof}

Combining Lemma $\ref{wakiM1}$ with Corollary $\ref{M1}$ we obtain the following proposition:

\begin{proposition}
\label{newsig}
There exists an equivalence of categories 
$$\tilde{\sigma}:P_{I_{M_{1}}}(\flagvar{M_{1}})\iso P_{I_{G}}(\flagvar{G})$$
such that for any $w$ in $W_{M_{1}},$ $\tilde{\sigma}$ sends $L_{w}$ to $L_{w_{0}\overline{w}^{-1}w_{0}},$
( $\overline{w}$ is the anti-involution defined in Definition $\ref{antiauto}$).
Additionally for any $\mathcal{T}$ in $P_{I_{M_{1}}}(\flagvar{M_{1}})$  we have
$$\heckefuncr{H}(\mathcal{T},\mcI^{w_{0}!})\iso \heckefunc{G}(\tilde{\sigma}(\mathcal{T}),\mcI^{w_{0}!}).$$
At last, For any $\lambda$ a cocharacter of $M_{1}$, we have
$$\tilde{\sigma}(\Theta_{\lambda}\star L_{\tau !})\iso \Theta_{-w_{0}\lambda}\star L_{w_{0}\tau w_{0}!}.$$
\end{proposition}
In the case $n=m$, the anti-isomorphism reduces to Proposition $\ref{sig}$. 
\section{Weak geometric analogue of Jacquet Functors and compatibility with
Hecke functors}
\label{jacquet}
In this section we place ourselves in a more general setting. 
Let $G$ be a split reductive connected group over $\K$,  $T$ be the maximal  standard torus of $G$ and $B$ be the standard Borel subgroup $B$ in $G$ containing $T$. Denote by $I_{G}$ the corresponding Iwahori subgroup. Let $P$ be a parabolic subgroup of $G$ containing $B$ and $U$ its unipotent radical. Let $L$ be the Levi subgroup of $P$  isomorphic to $P/U.$ Let $M_{0}$ be a faithful representation of $G$, and let $M=M_{0}\otimes_{\K}\locring.$ Denote by $\mathcal{S}(M(F))$ the Schwartz space of locally constant functions with compact support on $M(F)$.  In the classical setting, an important tool is the Jacquet module $\mathcal{S}(M(F))_{U(F)} $ of coinvariants with respect to $U(F).$ We will define a weak analogue of Jacquet functors in the geometric setting. Let $V_{0}$ be a $P$-stable subspace of $M_{0}$ endowed with a trivial action of $U$. Set $V=V_{0}\otimes_{\K}\locring$, we have a surjective map of $L(F)$-representations 
$$\mathcal{S}(M(F))_{U(F)}\longrightarrow \mathcal{S}(V(F))$$
given by restriction under the inclusion $V(F)\hookrightarrow M(F).$ We will geometrize the composition 
$$\mathcal{S}(M(F))\longrightarrow\mathcal{S}(M(F))_{U(F)}\longrightarrow \mathcal{S}(V(F))$$
at the Iwahori level.  Recall that the geometric version of the $I_{G}$-invariants of the Schwartz space $\mathcal{S}(M(F))^{I_{G}}$ constructed in \cite{BFH2} is $D_{I_{G}}(M(F))$.
We will define Jacquet functors 
$$J_{P}^{*}, J_{P}^{!}:D_{I_{G}}(M(F))\longrightarrow D_{I_{L}}(V(F))$$
which are exchanged by Verdier duality.  While the functor $J_P^{*}$ should be thought of as the classical Jacquet functor defined in representation theory, the functor $J^!_P$ has no analogue at the level of functions. 
The functor $J_{P}^{*}$ is a key object in the proof of Howe correspondence at the unramifed level for dual pairs $(\GL_n,\GL_m)$ \cite[\S 5]{Lysenko1}.  We  will show that  in the Iwahori case, geometric Jacquet functors commute  with the action of Hecke functors $\heckefunc{G}$. This construction extends to the Iwahori case  the one done in \cite[Corollary 3]{Lysenko1} at the unramifed level. 
In the classical setting the Jacquet functors of the Weil representations have been studied in \cite{Rallis} and \cite{Kudla}. The key ingredient to prove that geometric Jacquet functors commute  with the action of Hecke functors in the unramified case is the hyperbolic localization due to Braden \cite{Braden}. There is also an algebraic construction of geometric Jacquet functors due to Emerton-Nadler-Vilonen in the case of the real reductive groups by means of D-modules and nearby cycles on the flag variety \cite{ENV}. The geometric interpretation underlying all these constructions seems to be the same. 
\par\medskip
Let $I_{P}$ be the preimage of the Borel subgroup $B$ under the map $P(\locring)\to P.$ Denote by $B_{L}$ the image of $B$ in $L$. It is a Borel subgroup of $L$.  In the same way consider the map $L(\locring)\to L$ and denote by $I_{L}$ the preimage of $B_{L}$ under this map in $L(\locring)$. Hence $I_{L}$ is an Iwahori subgroup of $L(F).$ Finally we have a diagram 
$$I_{L}\longleftarrow I_{P}\hookrightarrow I_{G},$$
where the first map is induced by the natural projection $P(\locring)\to L(\locring).$ According to \cite{BFH2}, the categories $D_{I_{G}}(M(F))$ and $D_{I_{L}}(V(F))$ are well-defined. We are going to define the Jacquet  functors

$$J_{P}^{*}, J_{P}^{!}:D_{I_{G}}(M(F))\longrightarrow D_{I_{L}}(V(F)).$$

Let $N,r$ be two integers such that $N+r\geq 0.$ Set $V_{N,r}=t^{-N}V/t^{r}V.$ Denote by $i_{N,r}$ the natural closed embedding of $V_{N,r}$ in $M_{N,r}.$ For any $s\geq 0,$ let $K_{s}$ be the quotient of $I_G$ by the kernel of the map  $G{\locring})\to G(\locring/t^{s}\locring).$ Let $I_{P,s}$ denote the image of $I_{P}$ under the inclusion $$I_{P}\hookrightarrow P(\locring)\longrightarrow P(\locring/t^{s}\locring).$$
Similarly let $I_{L,s}$ be the image of $I_{L}$ under $L(\locring)\to L(\locring/t^{s}\locring)$.
We have the following diagram 

 \[
\xymatrix @R=1cm @C=1cm{
L(\locring/t^{s}\locring)  & P(\locring/t^{s}\locring) \ar[l] \ar[r]& G(\locring/t^{s}\locring) \\
I_{L,s} \ar[u]&  I_{P,s}\ar[l] \ar[r]  \ar[u] & K_{s}. \ar[u]\\
}
\]

For $s\geq N+r$, we obtain a digram of stack quotients
 \[
\xymatrix @R=1cm @C=1cm{
I_{P,s}\backslash V_{N,r} \ar[r]^{i_{N,r}} \ar[d]_{q} & I_{P,s}\backslash M_{N,r} \ar[r]^{p} & K_{s}\backslash M_{N,r} \\
I_{L,s}\backslash V_{N,r},\\
}
\]

where $p$ comes from the closed  inclusion  $I_{P,s}\hookrightarrow I_{G,s}.$
Set $a$  equal to $\dim M_{0}-\dim V_{0}.$
%\begin{proposition}
% There exists  two well-defined  functors 
%\begin{equation}
%\label{jfunc} 
%J_{P}^{*}, J_{P}^{!}:D_{I_{G}}(M(F))\longrightarrow D_{I_{L}}(V(F)).
%\end{equation} 
%\end{proposition}

 For any $s\geq N+r$, we have the following functors:
$$J_{P,N,r}^{*}, J_{P,N,r}^{!}:D_{K_{s}}(M_{N,r})\longrightarrow D_{I_{L,s}}(V_{N,r})$$
defined by 
$$q^{*}\circ J_{P,N,r}^{*}[\dim.\mathrm{rel}(q)]=(i_{N,r})^{*}p^{*}[\dim.\mathrm{rel}(p)-ra]$$
$$q^{*}\circ J_{P,N,r}^{!}[\dim.\mathrm{rel}(q)]=(i_{N,r})^{!}p^{*}[\dim.\mathrm{rel}(p)+ra].$$
The sequence 
$$1\longrightarrow U(\locring/t^{s}\locring)\longrightarrow I_{P,s}\longrightarrow I_{L,s}\longrightarrow 1$$
is exact. Hence the functor 
$$q^{*}[\dim.\mathrm{rel}(q)]:D_{I_{L,s}}(V_{N,r})\longrightarrow D_{I_{P,s}}(V_{N,r})$$
 is an equivalence of categories and exact for perverse $t$-structure. The functors $J_{P,N,r}^{*}$ and $J_{P,N,r}^{!}$ are well-defined.  They are compatible with the transition functors in the ind-system of categories defining $D_{I_{G}}(M(F))$ and $D_{I_{L}}(L(F))$ defined in \cite{BFH2}. By the taking the inductive 2-limit, we obtain the two well-defined functors $J_{P}^{*}$ and $J_{P}^{!}$  which do not depend on the choice of a section of $P\to P/U.$ The Verdier duality functor  $\mathbb{D}$ exchanges $ J_{P}^{*}$ and $J_{P}^{!}$, i.e. we have canonically 
$$\mathbb{D}\circ J_{P}^{*}\iso J_{P}^{!}\circ \mathbb{D}.$$ 
\par\medskip
As in the case of the affine flag variety, we can define the $\K$-space quotient  $P(F)/I_{P}$ and define $\mathcal{F}\ell_{P}$  to be the sheaf associated to this presheaf in fpqc-topology. The space $\mathcal{F}\ell_{P}$ is an ind-scheme. Let $X$ be a projective smooth connected curve over the field $\K$. Let $x$ be a closed point in $X$ and $X^*$ be equal $X-\{ x\}$.
Denote by $\locring_{x}$ the completion of the local ring of $X$ at $x$ and by $F_{x}$ its field of fractions. We choose a local coordinate at the point $x$, denoted by $t$, and  we may identify $\locring_{x}=\K[[t]]$ and $F_{x}=\K((t))$. Let $D= \mathrm{Spec(\K[[t]])}$ and  $D^{*}=\mathrm{Spec}(\K((t))).$   Then $\mathcal{F}\ell_{P}$ classifies $(\mathcal{F}_{P},\beta, \epsilon),$ where $\mathcal{F}_{P}$ is a $P$-torsor on $D,$ the map $\beta$ is a trivialization of $\mathcal{F}_{P}$ over $D^{*}$,  and $\epsilon$ is a reduction of $\mathcal{F}_{P}\vert_{x}$ to a $B$-torsor. We have the diagram 

\begin{equation}
\label{tp}
\mathcal{F}\ell_{L}\overset{\mathfrak{t}_{L}}{\longleftarrow} \mathcal{F}\ell_{P}\overset{\mathfrak{t}_{P}} {\longrightarrow} \mathcal{F}\ell_{G},
\end{equation}

where $\mathfrak{t}_{P}$ (resp. $\mathfrak{t}_{L}$) is given by extension of scalars with respect to $P\hookrightarrow G$ (resp. $P\to L$). 
Let $\mathcal{F}\ell_{P,G}$ be the $P(F)$-orbit through $1$ in $\mathcal{F}\ell_{G}$ viewed as an ind-subscheme with a reduced scheme structure. The 
reduced ind-scheme $\mathcal{F}\ell_{P,\mathrm{red}}$ gives a stratification of $\mathcal{F}\ell_{P,G}.$ There is a Hecke action $\heckefunc{G}$ of $D_{I_{G}}
(\flagvar{G})$ on the category $D_{I_{G}}(M(F))$  and a Hecke action $\heckefunc{L}$ of $D_{I_{L}}(\flagvar{L})$ on $D_{I_{L}}(L(F))$. Our aim is to prove that the functor 
$J^*_{P}$ commutes partially with these Hecke actions. In the unramified setting there exists a geometric restriction functor from the category 
$P_{G(\mathcal{O})}(Gr_{G})$ to the category $P_{L(\mathcal{O})}(L(F))$ verifying some properties \cite[Proposition 4.3.3]{Gaitsgory-ein}. Denote $\check{G}$ (resp. $\check{L}$) the Langlands dual group of $G$ over $\qelbar$ (resp. $L$). One can define a restriction functor $\mathrm{Rep}(\check{G})\to \mathrm{Rep}(\check{L})$ with respect to the map $\check{L}\to \check{G}$. Then, the geometric restriction functor $\mathrm{gRes}:P_{G(\mathcal{O})}(Gr_{G}) \to P_{L(\mathcal{O})}(L(F))$ corresponds via Satake isomorphism \cite{Mirkovic} to $\mathrm{Rep}(\check{G})\to \mathrm{Rep}(\check{L})$. The slightly different version of this geometric restriction functor has been defined in \cite{Lysenko1} taking in consideration a factor $\G_m$ which corresponds to the maximal torus of Arthur's $\SL_{2}$. We are going to define the same kind of geometric restriction functor at the Iwahori level, i.e.,  
$$\mathrm{gRes}: D_{I_{G}}(\mathcal{F}\ell_{G})\longrightarrow D_{I_L}(\mathcal{F}\ell_{L}).$$

For $s_{1},s_{2}\geq 0,$ let $\Pstrat=P(F)\cap \Gstrat,$ and $\FlagstratP= \Pstrat/I_{P},$ , where $$\Gstrat:=\{g\in{G(F)}\vert t^{s_{1}}V\subset gV\subset t^{-s_{2}}V\}.$$  

The ind-scheme $\FlagstratP$ is a closed subscheme of $\mathcal{F}\ell_{P}$. Similarly we define 
 $$\Lstrat:=\{g\in{L(F)}\vert t^{s_{1}}V\subset gV\subset t^{-s_{2}}V\},$$
 and  $\FlagstratL=\Lstrat/I_{L}.$ Thus the map $\mathfrak{t}_{L}$ in $\eqref{tp}$ induces a morphism (denoted again by $\mathfrak{t}_{L}$) from $\FlagstratP$ to $\FlagstratL.$ For $s\geq s_{1}+s_{2}+1$ we have a diagram of stack quotients 
  
 \[
\xymatrix{
I_{L,s}\backslash (\FlagstratL) & I_{P,s}(\FlagstratL) \ar[l]_{q_{L}} \\
& I_{P,s}\backslash(\FlagstratP) \ar[u]_{\mathfrak{t}_{L}}\ar[d]^{\mathfrak{t}_{P}}\\
.K_{s}\backslash (\Flagstrat)& I_{P,s}\backslash \Flagstrat \ar[l]_{\xi}
}
\]

Moreover, the functor 
$$q_{L}^{*}[\textrm{dim.rel}(q_{L})]:D_{I_{L,s}}(\FlagstratL)\longrightarrow D_{I_{P,s}}(\FlagstratL)$$
is an equivalence of categories and exact for the perverse $t$-structure. 
For any perverse sheaf $K$  extension by zero from $\Flagstrat$ to $\flagvar{G},$  we may define $\mathrm{gRes}(K)$ by the isomorphism 
$$q_{L}^{*}\mathrm{gRes}(K)[\dim.\mathrm{rel}(q_{L})]\iso (\mathfrak{t}_{L!})\mathfrak{t}^{*}_{P}\xi^{*}K[\dim.\mathrm{rel}(\xi)].$$

\begin{lemma}
\label{Llambda}
For  any dominant cocharacter $\lambda$ of $G$, we have 
$$\mathrm{gRes}(L_{t^{\lambda}!})\iso L_{t^{\lambda}!}[-\langle\lambda,2(\check{\rho}_{G}-\check{\rho}_{L})\rangle],$$
where $\check{\rho}_{G}$ (resp. $\check{\rho}_{L}$) denote the half sum of positive roots of $G$ (resp. positive roots of $L$).
\end{lemma}
\begin{proof}
Let $U_{B}$ be the unipotent radical of $B$. The space $\mathcal{F}\ell^{t^{\lambda}}_{G}$ is the $U_{B}(\locring)$-orbit through $t^{\lambda}I_{G}$ on $\mathcal{F}\ell_{G}.$ Thus $L_{t^{\lambda}!}$ is the extension by zero from a connected component of $\mathcal{F}\ell_{P}.$  The map 
$U_{B}t^{\lambda}I_{P}/I_{P}\longrightarrow \flagvar{L}^{t^{\lambda}}$
 is a trivial affine fibration with  affine fibre of dimension 
 $\langle\lambda, 2(\check{\rho}_{G}-\check{\rho}_{L})\rangle$ and the result follows.
 
 \end{proof}

\begin{lemma}
For any $w$  in the finite Weyl group of $L$, we have
$$\mathrm{gRes}(L_{w})\iso L_{w},\hspace{2mm}\mathrm{gRes}(L_{w!})\iso L_{w!},\hspace{2mm}\mathrm{gRes}(L_{w*})\iso L_{w*}.$$
\end{lemma}

\begin{proof}
The $P$-orbit through $I_{P}$ gives a natural  closed subscheme $L/B_{L}\iso P/B\hookrightarrow \flagvar{P}.$
 For any $w$ in the finite Weyl group of $L,$  the double coset $BwB$ is contained in $P.$ Thus $L_{w!}$ initially defined  over $\mathcal{F}\ell_{G}$ is actually an extension by zero from a connected component of $\mathcal{F}\ell_{P}$. Hence $\mathrm{gRes}(L_{w!})\iso L_{w!}.$ Moreover, $\overline{BwB/B}$ in $G/B$ actually lies in $P/B,$ so $\mathrm{gRes}(L_{w})\iso L_{w}.$ The same result holds for $L_{w*}.$
  \end{proof}
 
\begin{theorem}
\label{proppri}
Let $\mathcal{T}$ be a perverse sheaf in $P_{I_{G}}(\mathcal{F}\ell_{G})$ which is an extension by zero from a connected component of $\mathcal{F}\ell_{P}$, and  $\mathcal{K}$ be in $D_{I_{G}}(M(F))$ then we have 
$$J_{P}^{*}\overset{\leftarrow}{H}_{G}(\mathcal{T},\mathcal{K})\iso \overset{\leftarrow}{H}_{L}(\mathrm{gRes}(\mathcal{T}),J_{P}^{*}\mathcal{K})[\langle \lambda, \check{\nu}-\check{\mu} \rangle],$$
where $\lambda$ is the cocharacter  whose image in $\pi_{1}(L)$ is $\theta,$ $\check{\nu}$ is the character by which $L$  acts on $\det(V_{0})$ and $\check{\mu}$ is the character by which $G$ acts on $\det(M_{0})$.
 
\end{theorem}

\begin{proof}
The connected components of $\mathcal{F}\ell_{P}$ are indexed by $\pi_{1}(L).$ For $\theta$ in $\pi_{1}(L),$ denote by $\mathcal{F}\ell_{P}^{\theta}$ for the corresponding connected component which is the preimage of $\flagvar{L}^{\theta}$ under the map $\mathfrak{t}_L$ defined in $\eqref{tp}$. 
\par\smallskip
Let $s_{1}$ and $s_{2}$ be two non negative  integers and $\mathcal{T}$ be the extension by zero from $\FlagstratP^{\theta}=\FlagstratP\cap \mathcal{F}\ell_{P}^{\theta}.$
For $N+r\geq 0$ and $s\geq \mathrm{max}\{N+r,s_{1}+s_{2}+1\}$ consider the diagram

 \[
\xymatrix @R=2cm @C=1cm{
&V_{N,r}\times\Pstrat \ar[d]^{q_{P}} \ar[r]^{act} &V_{N+s_{1},r-s_{1}} \ar[d]^{q_{U}} \\
V_{N,r} \ar[d]^{i_{N,r}} & V_{N,r}\times \FlagstratP^{\theta} \ar[l]_-{pr} \ar[r]^{act_{q,P}} \ar[d]^{i_{N,r}\times id} & I_{P,s}\backslash V_{N+s_{1},r-s_{1}} \ar[d]^{i_{N+s_{1},r-s_{1}}}\\
M_{N,r} \ar[d] & M_{N,r}\times \FlagstratP^{\theta} \ar[l]_-{pr} \ar[r]^{act_{q,P}}\ar[d] & I_{P,s}\backslash M_{N+s_{1},r-s_{1}} \ar[d]^{p}\\
M_{N,r}& M_{N,r}\times \Flagstrat^{\theta} \ar[l]_-{pr} \ar[r]^{act_{q}}& K_{s}\backslash M_{N+s_{1},r-s_{1}},
}
\]
where  the map $act$ sends $(m,p)$ to $p^{-1}m$, the map $q_{P}$ sends $(m,p)$ to $(m,pI_{P})$, and $q_{U}$ is the stack quotient under the action of $I_{P,s}.$

Moreover, the second line of this diagram fits in the following diagram
 \[
\xymatrix @R=2cm @C=1cm{
V_{N,r}\times\FlagstratP^{\theta}\ar[r]^{act_{q,P}} \ar[d]^{id\times \mathfrak{t}_{L}} & I_{P,s}\backslash V_{N+s_{1},r-s_{1}} \ar[d]^{q}\\
V_{N,r}\times \FlagstratL^{\theta}\ar[r]^{act_{q,P}} & I_{L,s}\backslash V_{N+s_{1},r-s_{1}}.
}
\]

At the level of reduced ind-schemes the map $\FlagstratP^{\theta}\longrightarrow \flagvar{G}$ is a locally closed embedding, thus the perverse sheaf $\mathcal{T}$ may be viewed as a complex over $\FlagstratP^{\theta}$. For a given $\mathcal{K}$ and  large enough $N,r,$ by definition, up to a shift independent of $\mathcal{K}$ and $\mathcal{T}$  we have
$$\heckefunc{G}(\mathcal{T},K)\iso pr_{!}(act^{*}_{q,P}(K)\otimes pr_{2}^{*}(\mathcal{T})),$$
where $pr:K_{s}\backslash (M_{N,r}\newtimes \Flagstrat)\longrightarrow K_{s}\backslash \Flagstrat$ is defined in Appendix $\ref{Appendix}$. Thus by definition of $\mathrm{gRes}$ and $J_{P}^{*}$ and the commutativity of the  diagram above, we get 
$$J_{P}^{*}\overset{\leftarrow}{H}_{G}(\mathcal{T},\mathcal{K})\iso \overset{\leftarrow}{H}_{L}(\mathrm{gRes}(\mathcal{T}),J_{P}^{*}\mathcal{K})[?],$$

To determine the shift, one may consider the following special case where $\mathcal{K}$ is the constant perverse sheaf $I_{0}$ on $M$ and $\mathcal{T}$ equals $L_{t^{\lambda!}},$ where  $\lambda$ is a dominant cocharacter of $G$.  For  $N,r$ large enough, we have the following diagram 
\begin{equation} 
\label{40}
\xymatrix {
M_{N,r} & M_{0,r}\newtimes\flagvar{P}^{t^{\lambda}} \ar[l]_{\alpha_{M}}\\
V_{N,r}\ar[u]^{i_{N,r}} & V_{0,r}\newtimes\flagvar{P}^{t^{\lambda}} \ar[u] \ar[l]_{\alpha_{V}}.
}
\end{equation}

 Remind that $\flagvar{P}^{t^{\lambda}}$ is the $U_{B}(\locring)$-orbit through 
 $t^{\lambda} I_{P}$ in $\flagvar{P}.$ The scheme 
 $M_{0,r}\newtimes\flagvar{P}^{t^{\lambda}}$ is the scheme classifying  pairs 
 $(gI_{P},m)$, where $gI_{P}$ lies in $\flagvar{P}$ and  $m$ is an element of $gM/t^{r}M.$ Similarly 
 the scheme $V_{0,r}\newtimes \flagvar{P}{t^{\lambda}}$ is the scheme classifying  pairs 
 $(gI_{P},v),$ where $gI_{P}$ lies in $\flagvar{P}^{t^{\lambda}}$ and $v$  is an element of  $gV/t^{r}V.$  For large enough $r,$ we have $gV\cap t^{r}M =t^{r}V.$ So the right 
 vertical arrow in Diagram $\eqref{40}$ is a closed immersion. Denote by $\mathrm{IC}$ 
 the IC-sheaf of $M_{0,r}\newtimes \flagvar{P}^{t^{\lambda}}.$
We have canonically 
$$\overset{\leftarrow}H_{G}(L_{t^{\lambda}!},I_{0})\iso \alpha_{M!}(\mathrm{IC}),$$
and additionally
$$\dim (M_{0,r}\newtimes \flagvar{P}^{t^{\lambda}})=\langle\lambda, 2\check{\rho}_{G}\rangle+r\dim M_{0}-\langle\lambda,\check{\mu}\rangle,$$ 

 Hence we have 
\begin{align}
\label{iso1}
 J_{P}^{*}\overset{\leftarrow}{H}_{G}(L_{t^{\lambda}!},I_{0})\iso & i_{N,r}^{*}\alpha_{M!}(\mathrm{IC})[-ra]\nonumber \\
  \iso & \alpha_{V!}\qelbar[\langle\lambda, 2\check{\rho}_{G}\rangle+r\dim V_{0}-\langle\lambda,\check{\mu}\rangle].
\end{align}

  The map $\alpha_{V}$ factors through 
  \begin{equation}
  \label{iso3}
  V_{0,r}\newtimes \mathcal{F}\ell_{P}^{t^{\lambda}}\longrightarrow V_{0,r}\newtimes \mathcal{F}\ell_{L}^{t^{\lambda}}\overset{\alpha_{L}}{\longrightarrow}V_{N,r},
  \end{equation}
where  the first map is a trivial affine fibration with an  affine fibre of dimension  $\langle\lambda,2(\check{\rho}_{G}-\check{\rho}_{L})\rangle.$
 This gives us 
$$
  \dim (V_{0,r}\newtimes \mathcal{F}\ell_{L}^{t^{\lambda}})=\langle\lambda,2\check{\rho}_{L}\rangle+r\dim V_{0}-\langle\lambda,\check{\nu}\rangle.$$

 By definition we have 
  \begin{equation}
  \label{iso2}
  \overset{\leftarrow }{H}_{L}(L_{t^{\lambda}!},I_{0})\iso \alpha_{L!}(\mathrm{IC}).
  \end{equation}
This gives us the desired shift as follows
\begin{align}
 J_{P}^{*}\overset{\leftarrow}{H}_{G}(L_{t^{\lambda}!},I_{0})  \iso &   \alpha_{V!}\qelbar[\langle\lambda, 2\check{\rho}_{G}\rangle+r\dim V_{0}-\langle\lambda,\check{\mu}\rangle]\\
  	\iso &  \alpha_{L!}\qelbar[\langle\lambda,\check{\nu}-\check{\mu}-2(\check{\rho}_{G}-\check{\rho}_{L})\rangle]\nonumber\\
  	\iso & \overset{\leftarrow}{H}_{L}(L_{t^{\lambda}!},I_{0})[\langle\lambda,\check{\nu}-\check{\mu}-2(\check{\rho}_{G}-\check{\rho}_{L})\rangle]\nonumber\\
  \iso & \overset{\leftarrow}{H}_{L}(\mathrm{gRes}(L_{t^{\lambda}!}),I_{0})[\langle\lambda,\check{\nu}-\check{\mu}\rangle]\nonumber,
 \end{align}
where the first isomorphism is due to $\eqref{iso1},$ the second is due to $\eqref{iso3},$  the third is due to $\eqref{iso2}$ and the last one is due to Lemma $\ref{Llambda}.$
This shift is compatible with \cite[Lemma 5]{Lysenko1}.
\end{proof}

Let $\delta_{U}:\mathbb{G}_{\m}\times M_{0}\to M_{0}$ be a linear action, whose fixed points set is $V_{0}.$ Assume $\delta_{U}$ contracts $M_{0}$ onto $V_{0}$. Let $r$ be a integer, denote by $\nu:\mathbb{G}_{m}\to L$ the cocharacter of the center of $L$ acting on $V$ by $x\to x^{r}$.
Now consider the case where  $\delta_{U}$ is the map sending $x$ in $\mathbb{G}_{\m}$ to $\nu(x)x^{-r}$.  For any $x$ in $\mathbb{G}_{\m}$ and $m$ in $M_{N,r}$ consider the action of $\mathbb{G}_{\m}$ on $M_{N,r}$ defined by $(x,m)=xm.$ Let $K$ be a $\mathbb{G}_{\m}$-equivariant perverse sheaf in $P_{I_{G}}(M(F))$ with respect to this action of $\G_{m}$ on $M_{N,r}.$ Then for any $w$ in $\weyl{G}$,  both $K$ and $\heckefunc{G}(L_{w},K)$ are $\mathbb{G}_{\m}$-equivariant with respect to the $\delta_{U}$-action on $M(F).$ We get a new version of Proposition $\ref{proppri}$ as follows: 
\begin{corollary}
\label{hyploc}
Let $K$ be a $\mathbb{G}_{\m}$-equivariant perverse sheaf in $P_{I_{G}}(M(F))$ for the $\delta_{U}$-action on $M_{N,r},$ for $N,r$ large enough. Assume that $\K$ admits a $\K^{'}$-structure for some finite subfield $\K^{'}$ of $\K$,
and as such is pure of weight zero. Then $J_{P}^{*}(K)$ is pure of weight zero.
\end{corollary}
 This is an analogue of \cite[Corollary 3]{Lysenko1} in the Iwahori case affirming that the geometric Jacquet functors preserve the pure preserve  sheaves of weight zero. 
 
 \section{Appendix A}
 \label{Appendix}
 The construction of Hecke functors has been done in \cite[\S 3]{BFH2}. We will recall here its main lines for sake of completeness. Let $G$ be a split connected reductive group over $\K$. Let $T_{G}$ be the maximal standard torus and $B_{G}$ be  the standard Borel subgroup containing $T_{G}$ in $G$. Denote by $I_{G}$ the corresponding Iwahori subgroup. Let $M_{0}$ be a  faithful finite-dimensional representation of $G$ and let $M=M_{0}\otimes _{\K}\mathcal{O}$. The definitions of the derived category  $D_{I_{G}}(M(F))$ of $\ell$-adic sheaves on $M(F)$ and the category $P_{I_{G}}(M(F))$ of $\ell$-adic perverse sheaves on $M(F)$ are given in \cite[\S $3$]{BFH2}.  For any two integers $N,r\geq 0$ with $N+r>0,$  set $M_{N,r}=t^ {-N}M/t^{r}M.$ The subgroup $G(\locring)$  acts on $M_{N,r}$ via its finite dimensional quotient $G(\locring/t^{N+r}\mathcal{O}).$ Denote by $I_{s}$ the kernel of the map $G(\locring)\longrightarrow G(\locring/t^ {s}\locring).$ The Iwahori subgroup $I_G$ acts on $M_{N,r}$ via its finite-dimensional quotient $I_G/I_{N+r}$. For $s>0$ denote by $K_s$ the quotient $I_G/I_{s}.$
 Let $s_1,s_2\geq 0$ and set 
\begin{equation}
\label{gstrat}
\Gstrat=\{ g\in G(F) \vert \hspace{2mm} t^{s_{1}}M\subset{gM}\subset{t^{-s_{2}}M} \}.
\end{equation}
Then $\Gstrat\subset G(F)$ is closed and stable under the left and right multiplication by $G(\locring).$ 
Further, $\Flagstrat=\Gstrat/ I_{G}$ is closed in $\flagvar{G}.$ For $s^{'}_{1}\geq s_{1}$ and $s^{'}_{2}\geq s_{2},$
 we have the closed embeddings ${}_{s_{1},s_{2}}\mathcal{F}\ell_{G}\hookrightarrow{{}_{s^{'}_{1},s_{2}^{'}}\mathcal{F}\ell_{G}}$
and the union of ${}_{s_{1},s_{2}}\mathcal{F}\ell_{G}$ is the affine flag variety $\flagvar{G}.$  The map sending $g$ to $g^{-1}$ yields an isomorphism between ${{}_{{}_{s_{1},s_{2}}}G(F)}$ and ${{}_{{}_{s_{2},s_{1}}}G(F)}.$ Denote by $\check\mu$ in $\check{X}^{+}$ the character by which $G$ acts on $\det(M_{0}).$  The connected components of the affine Grassmannian $Gr_{G}$ are indexed by the algebraic fundamental group $\pi_{1}(G)$ of $G$. For $\theta$  a cocharacter in $\pi_{1}(G)$, choose $\lambda$ in $X^{+}$ whose image in $\pi_{1}(G)$ equals $\theta.$ Denote by $Gr_{G}^{\theta}$ the connected component of $Gr_G$ containing $Gr_{G}^{\lambda}$. The affine flag manifold $\flagvar{G}$ is a fibration over $Gr_{G}$ with the typical fibre $G/B.$ Hence the connected components of the affine flag variety $\flagvar{G}$ are also indexed  by $\pi_{1}(G).$ For $\theta$ in $\pi_{1}(G),$ denote by $\flagvar{G}^{\theta}$ the preimage of $Gr_G^{\theta}$ in $\flagvar{G}$. Set $\Flagstrat^{\theta}=\flagvar{G}^{\theta}\cap{\Flagstrat}.$ According to \cite[Lemma 4.2]{BFH2}, There exists  an inverse image functor 
$$act_{q}^{*}:D_{I_G}(M(F))\times D_{I_G}(\flagvar{G}) \longrightarrow D_{I_G}(M(F)\times \flagvar{G})$$
which preserves perversity and is compatible with the Verdier duality in the following way:
for any $\mathcal{K}$ in $D_{I_{G}}(M(F))$ and $\mathcal{F}$ in $D_{I_{G}}(\flagvar{G})$ we have  $$\mathbb{D}(act_{q}^{*}(\mathcal{K},\mathcal{T}))\iso act_{q}^{*}(\mathbb{D}(\mathcal{K}),\mathbb{D}(\mathcal{T})).$$

Given $N,r,s_{1},s_{2}\geq 0$ with $r\geq{s_{1}}$ and $s\geq{\max\{N+r,s_{1}+s_{2}+1\}},$
one can define the following commutative diagram 
\begin{center}
\[
\xymatrix @R=2cm{
&& {M}_{{}_{N,r}}\times \Gstrat \ar[rr]^{act} \ar[d]_{q_G} && M_{{}_{N+s_{1},r-s_{1}}}\ar[d]^{q_{M}} \\
  {M}_{{}_{N,r}} \ar[d] &&  {M}_{{}_{N,r}}\times{\Flagstrat} \ar[d]  \ar[ll]_-{pr_{1}} \ar[rr]^{act_{q}}   && K_{s}\backslash{M_{N+s_{1},r-s_{1}}} \\
K_{s}\backslash {M}_{{}_{N,r}}  && K_s\backslash({{M}_{{}_{N,r}}\times\Flagstrat)}  \ar[ll]_-{pr} \ar[rru]^{act_{q,s}} \ar[rr] ^{pr_{2}} && K_s\backslash\hspace{1mm}(\Flagstrat)}
\]
\end{center}

The action map $act$ sends the couple $(v,g)$ to $g^{-1}v$. The maps  $pr_{1}$, $pr_2$ and $pr$ are projections. The map $q_{G}$ sends the couple $(v,g)$ to $(v,g{I_{G}}).$ All the vertical arrows are stack quotients for the action of the corresponding group. The group $K_{s}$ acts diagonally on $M_{N,r}\times \Flagstrat$ and the map $act_{q}$ is equivariant with respect to this action. This functor sends $(\mathcal{K},\mathcal{T})$ to 
$$K\tilde{\boxtimes}\mathcal{T}:=(act_{q,s}^{*}\mathcal{K})\otimes{pr_{2}^{*}}\mathcal{T}[\dim(K_s)-c+s_1\dim M_0]
$$
where $c$ equals $\langle\theta,\check\mu\rangle$ over $\Flagstrat^{\theta}.$  

For any $N,r,s_{1},s_{2}$ greater than  zero satisfying the condition $s\geq{\max\{N+r,s_{1}+s_{2}+1\}},$ consider the projection 
$$pr:K_s\backslash{(M_{N,r}\times{\Flagstrat})}\longrightarrow K_s\backslash M_{N,r}$$.

For any $\mathcal{K}$ in $D_{I_{{}_{G}}}(M(F))$ and $\mathcal{T}$ in $D_{I_{{}_{G}}}(\flagvar{G})$, the Hecke functor 
$$\overset{\leftarrow}{H_{G}}(\, ,\, ):D_{I_{G}}(\flagvar{G})\times D_{I_{G}\times I_{H}}(\Pi(F))\to D_{I_{G}\times I_{H}}(\Pi(F))$$  

is defined by 
$$
\heckefunc{G}(\mathcal{T}, \mathcal{K})=pr_{!}((K\tilde{\boxtimes}\mathcal{T}))
$$
Moreover, this functor is compatible with the convolution product on $D_{I_{G}}(\flagvar{G})$. Namely, given $\mathcal{T}_1, \mathcal{T}_2$ in $D_{I_G}(\flagvar{G})$ and $\mathcal{K}$ in $D_{I_G}(M(F))$, one has naturally
$$
\heckefunc{G}(\mathcal{T}_1, \heckefunc{G}(\mathcal{T}_2, \mathcal{K}))\iso \heckefunc{G}(\mathcal{T}_1\star  \mathcal{T}_2, \mathcal{K}).
$$ 

\paragraph{\textbf{An example of computation of Hecke functors }}
\label{Example}
Let $R,r\geq 0$ and $t^rM\subset V\subset t^{-R}M$ be an intermediate lattice stable under $I_G.$
Let $K\in{P_{I_{G}}}(M_{R,r})$ be a shifted local system on $V/t^{r}M\subset t^{-R}M/t^{r}M.$ We are going to explain the above construction explicitly in this case. Let $\mathcal{T}$ be in $D_{I_{G}}(\Flagstrat).$  Choose $r_{1}\geq r+s_{1}.$ If $g$ is a point in $\Flagstrat$ then $t^{r_{1}}M\subset gV.$ So we can define the scheme  $(V/t^ {r}M)\tilde{\times}\Flagstrat$
as  the scheme classifying pairs $(g\I{G},m)$ such that $gI_{G}$ is an element of $\Flagstrat$ and $m$ is in $(gV)/(t^{r_1}M).$ For a point $(m,g)$ of this scheme we have $g^ {-1}m$  in $V/t^{r}M.$ Assuming $s\geq R+r$ we get the digram 
$$M_{R+s_{2},r_{1}}\overset{p}{\longleftarrow}(V/t^{r}M)\tilde{\times}{\Flagstrat}\overset{act_{q,s}}{\longrightarrow} K_{s}\backslash (V/t^{r}M),$$
where $p$ is the map sending $(g\I{G},m)$ to $m$. For $gG(\locring)$ in ${Gr^{\theta}_G},$  the virtual dimension of $V/gV$ is $\langle\theta,\check{\mu}\rangle.$ The space $(V/t^ {r}M)\tilde{\times}\Flagstrat^{\theta}$ is locally trivial fibration over $\Flagstrat^{\theta}$ with fibre isomorphic to an affine space of dimension $\dim(V/t^{r_1}M)-\langle\theta,\check{\mu}\rangle.$ Since $K$ is a shifted local system, the tensor product 
$act_{q,s}^ {*}K \otimes pr_{2}^{*}\mathcal{T}$ is a shifted perverse sheaf. Let $K\tilde{\boxtimes}\mathcal{T}$ be the perverse sheaf $act_{q,s}^{*}K\otimes pr_{2}^{*}\mathcal{T}[\dim].$ The shift $[\mathrm{dim}]$ in the definition depends on the dimension of the connected component  and hence on $\check{\mu}$ as explained above and is such that the sheaf $act_{q,s}^{*}K\otimes pr_{2}^{*}\mathcal{T}[\dim]$ is perverse. Then  
$
\heckefunc{G} (\mathcal{T},K)=p_{!}(K\tilde{\boxtimes}\mathcal{T}).
$
%Consider the category $DP_{I_{G}}(\flagvar{G})=\oplus_{r\in{\Z}}P_{I_{G}}(\flagvar{G})[r]$ defined in [\S\ $\ref{perversesheaves}$, Chatper $\ref{chapterone}$]. The group $\mathbb{G}_{m}$ acts on $P_{I_{G}}(\flagvar{G})[r]$ by the character $x\to x^{-r}$.
%Let $\kappa:\check{L}\times \mathbb{G}_{m}\to \check{G}$ be the map whose first component is a Levi factor $\check{L}\overset{i_{\check{L}}}{\longrightarrow}\check{G}$ and the second is the composition
%$$\mathbb{G}_{m}\overset{2(\check{\rho}_{G}-\check{\rho}_{L})}{\longrightarrow} Z(\check{L})\hookrightarrow \check{L}\hookrightarrow \check{G},$$
%where $Z(\check{L})$ is the center of $\check{L}.$
%Then we may define the following functor similar to $\mathrm{gRes}$, corresponding to the map $\kappa$ :
%$$\mathrm{gRes}^{\kappa}:P_{I_{G}}(\flagvar{G})\longrightarrow DP_{I_{G}}(\flagvar{G})[r].$$
%This analogous to the functor $\mathrm{gRes}^{\kappa}$ constructed in [Corollary 2, \cite{Lysenko}] in the unramified case.

\section{Appendix B}
\label{Appendix2}
% First we will compute this action where $w$ is a reflection $(i,i+1)$ in the finite Weyl groupe of $H.$ We will denote by $\tau_{i}$ the permutation $(i,i+1).$ Then we will compute the action while $w=t^{\lambda}w_{0}$ where $\lambda=(1,0,\dots,0,-1)$.  This the unique affine reflection. The action of elements of length zero begin evident this defines the whole action of the Iwahori-Hecke algebra at the classical level.  
The aim of this appendix is to compute the complex $\heckefunc{H}(L_{\tau},\mcI^{\mu!})$ in the category $D_{I_{H}\times I_{G}}(\Pi(F))$ for any $L_{\tau}$ in $P_{I_{H}}(\flagvar{H})$. We will first consider the case of $\tau$ being a simple reflection the finite Weyl group of $H$ then for $\tau$ being the unique simple affine reflection in $\weyl{H}$. The action of length zero elements being obvious, this completes the action of simple objects of $P_{I_{H}}(\flagvar{H})$ on $\mathcal{I}^{\mu}$ in the derived category $D_{I_{H}\times I_{G}}(\Pi(F))$. For any point $hI_{H}$ in $\overline{\flagvar{H}^{\tau}}$ we write $U^{'}_{i}=hU_{i}$ and we fix a complete flag $U_{1}^{'}\subset \dots \subset U_{m}^{'}$ on $U^{'}/tU^{'}.$ Let  $\overline{\Pi}_{1}^{\mu}\newtimes \overline{\flagvar{H}^{\tau}}$ be the scheme classifying  pairs $(v,hI_{H}),$ where  $hI_{H}$ is  in $\overline{\flagvar{H}^{\tau}}$ and a $v$ is a map from $L^{*}$ to $U^{'}/tU_{'}$ such that $v(e_{i}^{*})\in U_{\nu(i)}^{'}$ for all $i=1,\dots,n.$ Let 
\begin{equation}
\label{ferme}
\pi: \overline{\Pi}_{1}^{\mu}\newtimes \overline{\flagvar{H}^{\tau}}\to \Pi_{0,1}
\end{equation}
be the map sending $(v,hI_{H})$ to $v$. By definition we have 
$$\heckefunc{H}(L_{\tau},\mcI^{\mu})\iso \pi_{!}(\qelbar\bt L_{\tau}).$$

Let $\Pi_{1}^{\mu}\newtimes \flagvar{H}^{\tau}$ be the open subscheme in  $\overline{\Pi}_{1}^{\mu}\newtimes \overline{\flagvar{H}^{\tau}}$ consisting of pairs $(v,hI_{H})$ in $\Pi_{1}^{\mu}\newtimes \flagvar{H}^{\tau}$ such that $v(e_{i}^{*})\notin U_{\nu(i)-1}^{'}$ for all $i=1,\dots,n$. If $\pi^{0}$ is the restriction of $\pi$ to the open subscheme $ \Pi_{1}^{\mu}\newtimes \flagvar{H}^{\tau},$ then we have 
$$\heckefunc{H}(L_{\tau!},\mcI^{\mu !})\iso \pi_{!}^{0}(\qelbar\bt L_{\tau}).$$

For  $1\leq i<m$ we will denote by $\tau_{i}$ the simple reflection  $(i,i+1)$ in $W_{H}.$
\begin{proposition}
\label{gorthen0}
Let $i$ be an integer such that  $1\leq i< m$
 \par\smallskip
\begin{enumerate}
\item If $i\notin{I_{\mu}}$ then the complex $\overset{\leftarrow}{H}_{H}(L_{\tau_{i}},\mathcal{I}^{\mu})$ is canonically isomorphic to  $\mathcal{I}^{\mu}\otimes \mathrm{R}\Gamma(\mathbb{P}^{1},\qelbar)[1].$ 
\item If $i\in{I_{\mu}}$ then 
\[\overset{\leftarrow}{H}_{H}(L_{\tau_{i}},\mathcal{I}^{\mu})\iso \left\lbrace
\begin{array}{rcr}
           \mathcal{I}^{\tau_{i}\circ \nu}\oplus \mathcal{I}^{\tau_{i-1}\circ \nu} & \qquad \mathrm{if}\quad i>1\hspace{2mm} and \hspace{2mm} i-1\notin{I_{\mu}}\\
            \mathcal{I}^{\tau_{i}\circ \nu}\oplus \mathrm{\IC}(Y^{''})& \qquad otherwise,\\
\end{array}
\right.\]
\end{enumerate}
where $Y^{''}$ is a specific locally closed subscheme of $\Pi_{0,1}$ (whose construction will be given in the proof below).
\end{proposition}
  
  \begin{proof} 
  The scheme $\overline{\flagvar{H}^{\tau}}$ is the projective space of lines in $U_{i+1}/U_{i-1}.$ 
\begin{enumerate}
\item Consider the projection $p_{H}:\flagvar{H}\to Gr_{H}.$  Then $p_{H !}(L_{\tau_{i}})$ is canonically isomorphic to $\mathrm{R\Gamma}(\mathbb{P}^{1},\qelbar)[1].$ 
%In other words 
%$$p_{H !}(L_{\tau_{i}})\iso \qelbar [-1]\oplus \qelbar[1].$$
This implies that 
$$\heckefunc{H}(L_{\tau_{i}},\mcI^{\mu})\iso \mcI^{\mu}\otimes \mathrm{R}\Gamma(\mathbb{P}^{1},\qelbar)[1].$$

\item Let us describe the image of the map $\pi$ in $\eqref{ferme}$.  It  is contained in the closed subscheme $Y^{'}$ of $\Pi_{0,1}$ given by the following two conditions:
\begin{enumerate}
\item[a)] For $j\neq \nu^{-1}(i),$ $v(e_{j}^*)\in{U_{\nu(j)}}.$ 
\item[b)] For $j=\nu^{-1}(i),$ $v(e_{j}^{*})\in{U_{i+1}}.$
\end{enumerate}
Let $Y^{''}$  be the closed subscheme in $Y^{'}$ consisting of elements $v$ such that $v(e_j^{*})$ belongs to $U_{i-1}$ if $j=\nu^{-1}(i).$
Over a point of $Y^{''}$ the fibre of the map $\eqref{ferme}$ is $\mathbb{P}^{1}$ whence over a point of $Y^{'}-Y^{''}$ the fibre of the map  $\eqref{ferme}$ is a point. Thus 
$$\overset{\leftarrow}{H}_{H}(L_{\tau},\mathcal{I}^{\mu})\iso \mathrm{\IC}(Y^{'})\oplus \mathrm{\IC}(Y^{''}).$$
Now we need to identify $Y^{'}$ and $Y^{''}.$
\par\medskip
If $i+1\notin I_{\mu},$  let $I_{\mu^{'}}$ be the subset of  $\{1,\dots,m\}$  obtained from $I_{\mu}$ by throwing $i$ away and adding $i+1.$ In this case $Y^{'}$ is isomorphic to  $\overline{\Pi}_{0,1}^{\mu^{'}}$ so $\IC(Y^{'})$ is canonically isomorphic to $\mathcal{I}^{\mu^{'}}.$ 
Now let $\nu^{'}:\{1,\dots ,n\}\longrightarrow \{1,\dots, m\}$ be strictly decreasing with image $I_{\mu^{'}}.$ Considering $\nu^{'}$ as an element  of $S_{n,m}$, we get that $\Pi_{0,1}^{\nu^{'}}=\Pi_{0,1}^{\mu^{'}}.$
\par\medskip
If $i+1\in{I_{\mu}},$ let $\nu^{'}:\{1,\dots,n\}\longrightarrow \{1,\dots,m\}$ be the map $\nu$ composed with the permutation $\tau_{i}$.  The image of $\nu^{'}$ is the subset $I_{\mu}$(but $\nu^{'}$ is not strictly decreasing). Viewing $\nu^{'}$ as an element of $S_{n,m}$ enables us to identify  $Y^{'}$ with the closure of the $I_{H}\times I_{G}$-orbit $\Pi_{0,1}^{\nu^{'}}.$
\par\medskip
 Thus, in both cases $\nu^{'}=\tau_{i}\circ \nu$ is the composition $\{1,\dots,n\}\overset{\nu}{\longrightarrow }\{1,\dots , m\}\overset{\tau_{i}}{\longrightarrow}\{1,\dots,m\}$ and $\mathrm{\IC}(Y^{'})$ is isomorphic to $\mathcal{I}^{\tau_{i}\circ \nu}.$
\par\medskip
If $i>1$ and $i-1\notin I_{\mu}$ then $\tau_{i-1}\circ \nu :\{1,\dots, n\}\longrightarrow \{1,\dots ,m\}$ is strictly decreasing and $Y^{''}$ is isomorphic to the closure of $\Pi_{0,1}^{\tau_{i-1}\circ \nu}.$ So we get $\mathrm{\IC}(Y^{''})=\mathcal{I}^{\tau_{i-1}\circ \nu}.$  This proves the assertion.
\end{enumerate}

\end{proof}
 
\begin{proposition}
\label{gorthen1}
Let $i$ be an integer such that $1\leq i <m$
\par\smallskip
\begin{enumerate}
\item If  neither $i$ nor  $i+1$ is in  $I_{\mu}$ then 
$$\overset{\leftarrow}{H}_{H}(L_{\tau_{i}},\mathcal{I}^{\mu!})\iso \mathcal{I}^{\mu!}\otimes \mathrm{R}\Gamma(\mathbb{P}^{1},\qelbar)[1]$$

\item If $i$ is not in $I_{\mu}$ and  $i+1$ is in $I_{\mu}$ then the composition $\tau_{i}\circ \nu$ is again strictly decreasing and  there is a distinguished triangle
$$\mcI^{\mu!}[-1]\longrightarrow \overset{\leftarrow}{H}_{H}(L_{\tau_{i}},\mathcal{I}^{\mu!})\longrightarrow \mathcal{I}^{(\tau_{i}\circ \nu)!}\overset{+1}{\longrightarrow}$$

\item If $i$ is in $I_{\mu}$ and  $i+1$ is not an element of $I_{\mu}$ then there is a distinguished triangle 

$$\mathcal{I}^{(\tau_{i}\circ\nu)!}\longrightarrow \overset{\leftarrow}{H}_{H}(L_{\tau_{i}},\mathcal{I}^{\mu!})\longrightarrow\mathcal{I}^{\nu!}[1]\overset{+1}{\longrightarrow}.$$ 

\item If  both $i$ and $i+1$ are in $I_{\mu}$ then there is a distinguished triangle
$$\mathcal{I}^{(\tau_{i}\circ \nu)!}\longrightarrow \overset{\leftarrow}{H}_{H}(L_{\tau_{i}},\mathcal{I}^{\mu!})\longrightarrow\mathcal{I}^{\mu!}[1]\overset{+1}{\longrightarrow}.$$
\end{enumerate}
\end{proposition}

\begin{proof}
\begin{enumerate}

\item
This is straightforward as in Proposition $\ref{gorthen0}.$

\item 

Let $Y^{'}$ be the locally closed subscheme of $ \Pi_{0,1}$ given by the conditions :
\begin{enumerate}
\item[a)] For $1\leq j\leq n,$ $v(e_{j}^{*})\in{U_{\nu (j)}}.$ 
\item[b)] For $j\neq\nu^{-1}(i+1),$ $v(e_{j}^{*})\notin{U_{\nu (j)-1}}.$ 
\item[c)] For $j=\nu^{-1}(i+1),$ $v(e_{j}^{*})\notin {U_{\nu( j)-2}}.$
\end{enumerate}
The scheme $Y{'}$  is the union of two $I_{H}\times I_{G}$-orbits corresponding to $\nu$ and $\tau_{i}\circ \nu.$ Moreover, we have $$\dim(\Pi_{0,1}^{\nu})=1+\dim(\Pi_{0,1}^{\tau_{i}\circ \nu}).$$ 
Hence
$$\heckefunc{H}(L_{\tau_{i}},\mcI^{\mu !})\iso \IC(Y^{'})[-1]$$
and the assertion follows. 
\item 
Let $Y^{'}$ be the scheme classifying elements $v$ in $\Pi_{0,1}$ such that :
\begin{enumerate}
\item[a)] For $j\neq\nu^{-1}(i),$ $v(e_{j}^{*})\in U_{\nu( j)}$ and $v(e_{j}^{*})\notin U_{\nu( j)-1}.$
\item[b)] For $j=\nu^{-1}(i),$ $v(e_{j}^*)\in{U_{i+1}}$ and $v(e_{j}^{*})\notin{U_{i-1}}.$ 
\end{enumerate}

Then $Y^{'}$ is the union of two orbits $\Pi_{0,1}^{\nu}$ and $\Pi_{0,1}^{\tau_{i}\circ \nu}.$ Thus $\overset{\leftarrow}{H}_{H}(L_{\tau_{i}},\mcI^{\mu !})$ is the extension by zero of $\IC(Y^{'})$ from $Y^{'}$ to $\Pi_{0,1}.$ Hence we have a distinguished triangle 
$$\mathcal{I}^{(\tau_{i}\circ \nu)!}\longrightarrow \mathrm{\IC}(Y^{'})\longrightarrow \mathcal{I}^{\nu!}[1]\overset{+1}{\longrightarrow}.$$
\item 
Let $Y^{'}$ be a locally closed subscheme $\Pi_{0,1}$ be the scheme given by the conditions:
\begin{enumerate}
\item [a)] For $\nu(j)\neq i, i+1,$  $v(e_{j}^{*})\in{U_{\nu( j)}\setminus U_{\nu (j)-1}}.$  
 \item [b)] For $j=\nu^{-1}(i),$ $v(e_{j}^{*})$ and $v(e_{j-1}^*)$ belong to $U_{i+1},$ and their classes modulo  $U_{i-1}$ form a basis of $U_{i+1}/U_{i-1}.$
 \end{enumerate}
 Then $\overset{\leftarrow}{H}_{H}(L_{\tau_{i}},\mathcal{I}^{\mu})$ is isomorphic to $\IC(Y^{'})$ extended by zero to $\Pi_{0,1}.$ The scheme $Y^{'}$ is the union of two orbits $\Pi_{0,1}^{\nu}$ and $\Pi_{0,1}^{\tau_{i}\circ \nu},$ and we have a distinguished triangle 
 $$\mathcal{I}^{(\tau_{i}\circ \nu)!}\longrightarrow \IC(Y^{'})\longrightarrow \mathcal{I}^{\nu!}[1]\overset{+1}{\longrightarrow}.$$
\end{enumerate} 
\end{proof}

\begin{proposition}
\label{grothen3}
Let $w$ be the affine simple reflection, i.e. $w=t^{\lambda}\tau$ where $\lambda=(-1,0\dots,0,1)$ and $\tau=(1,m)$ is the longest element of $W_{H}.$ 
\par\smallskip 
 \begin{enumerate}
  \item If  neither $1$ nor $m$ lies in $I_{\mu}$ then 
  $$\overset{\leftarrow}{H}_{H}(L_{w},\mathcal{I}^{\mu})\iso \mathcal{I}^{\mu!}\otimes R\Gamma(\mathbb{P}^{1},\qelbar)[1].$$
  \item If $1$ is not in $I_{\mu}$ and $m$ lies in $I_{\mu},$  let $\lambda^{'}=(-1,0,\dots,0),$ and $w'=(\lambda^{'},\tau\circ \nu)$ be an element of $X_{G}\times S_{n,m}.$ There is a distinguished triangle 
  $$\mathcal{I}^{w^{'}!}\longrightarrow \overset{\leftarrow}{H}_{H}(L_{w},\mathcal{I}^{\mu!})\longrightarrow\mathcal{I}^{\mu!}[1] \overset{+1}{\longrightarrow}.$$
  \item
  If $1$ is in $I_{\mu},$ and  $m$ is not in $I_{\mu},$  let $\lambda^{'}=(0,\dots,0,1)$, and  $w^{'}=(\lambda^{'},\tau\circ \nu)$ be an element of $X_{G}\times S_{n,m},$ then there is a distinguished triangle
  $$\mathcal{I}^{\mu!}[-1]\longrightarrow \overset{\leftarrow}{H}_{H}(L_{w},\mathcal{I}^{\mu!)}\longrightarrow \mathcal{I}^{\mu!}[1]\overset{+1}{\longrightarrow}.$$
  \item If $1$ and $m$ are both in $I_{\mu},$ let $\lambda^{'}=(-1,0,\dots,0,1)$ and  $w^{'}=(\lambda^{'},\tau\circ \nu)$ then there is a distinguished triangle 
  $$\mathcal{I}^{w^{'}!}\longrightarrow\heckefunc{H}(L_{w},\mathcal{I}^{\mu!})\longrightarrow\mathcal{I}^{\mu!}[1]\overset{+1}{\longrightarrow}.$$ 
 \end{enumerate}
\end{proposition}

\begin{proof}
Denote by $U_{-1}\subset U_{0}\subset U_{1}\subset \dots \subset U_{m}=U\subset U_{m+1}$ the standard flag of lattices in $U(F)$. Assume $m>1.$ A point $hI_{H}$ in $\overline{\flagvar{H}^{w}}$ is given by a line $U_{0}^{'}/U_{-1}$ in $U_{1}/U_{-1}.$ We set $U_{m}^{'}=t^{-1}U^{'}_{0}.$ Let $\Pi^{\mu}\newtimes \overline{\flagvar{H}^{w}}$ be the scheme classifying pairs $(v,hI_{H}),$ where $hI_{H}$ is  in  $\overline{\flagvar{H}^{w}}$ and $v$ is a map from $L^{*}$ to $U_{m+1}/U_{-1}$
verifying:
\begin{enumerate}
\item [a)] For $\nu(j)\neq m$, $v(e_{j}^{*})\in U_{\nu(j)}.$
\item [b)] For  $\nu(j)\neq 1$, $v(e_{j}^{*})\in U_{\nu(j)}-U_{\nu(j)-1}.$
\item [c)] For $m\in I_{\mu}$, $v(e_{1}^{*})\in U_{m}^{'}-U^{'}_{m-1}$.

(The condition $m\in I_{\mu}$  is equivalent to $\nu(1)=m).$
\item[d)] If $1\in I_{\mu}$ then $v(e_{n}^{*})\in U_{1}-U^{'}_{0}.$

(The condition $1\in I_{\mu}$ is equivalent to $\nu(n)=1).$
\end{enumerate}
Now let 
\begin{equation}
\label{map1m}
\pi: \Pi^{\mu}\newtimes \overline{\flagvar{H}^{w}}\longrightarrow \mathrm{Hom}_{\locring}(L^{*},U_{m+1}/U_{-1})
\end{equation}
be the projection sending a couple $(v,hI_{H})$ to $v.$  The scheme  $\Pi^{\mu}\newtimes \overline{\flagvar{H}^{w}}$ is smooth. Write $\mathrm{\IC}$ for the intersection cohomology sheaf of $\Pi^{\mu}\newtimes \overline{\flagvar{H}^{w}}$. The sheaf $\IC$ is nothing but the constant sheaf shifted to be perverse. Then 
$$\heckefunc{H}(L_{w},\mcI^{\mu!})\iso \pi_{!}(\IC).$$
We can now prove the assertions.
\begin{enumerate}
\item This is straightforward as in Proposition $\ref{gorthen0}.$
\item The space $\mathrm{Hom}_{\locring}(L^{*},U_{m+1}/U_{1})$ is an $\locring$-module on which $t$ acts trivially, hence it is a vector space. By definition $\heckefunc{H}(L_{w},\mathcal{I}^{\mu !})$ may be considered as a complex on $\mathrm{Hom}_{\locring}(L^{*},U_{m+1}/U_{1}).$ Let $Y^{'}$ be the subscheme of $\mathrm{Hom}_{\locring}(L^{*},U_{m+1}/U_{1})$ given by the conditions:
\begin{enumerate}
\item[a)] For $j>1$, $v(e_{j}^{*})\in {U_{\nu(j)}-U_{\nu(j)-1}}.$
\item[b)] The vector $v(e_{1}^{*})$ does not vanish in $U_{m+1}/U_{m-1}.$
\end{enumerate}
Then $\heckefunc{H}(L_{w},\mathcal{I}^{\mu!})$ is isomorphic to $\IC(Y)^{'}$ extended by zero to $\mathrm{Hom}_{\locring}(L^{*},U_{m+1}/U_{1}).$
 The subscheme $Y^{'}$ is the union of two $I_{H}\times I_{G}$-orbits, the closed orbit corresponds to $\Pi_{N,r}^{\mu}$ and the open orbit passes through maps $v$ given by 
  $$v(e_{n}^{*})=u_{\nu (n)},\dots,v(e_{2}^{*})=u_{\nu( 2)},v(e_{1}^{*})=t^{-1}u_{1}.$$
  The map $v$ can be written as  follows: 
  \begin{enumerate}
  \item [a)] For $j\in{I_{\mu}}-\{m\},$ $v(u_{j}^{*})=e_{\mu( j)}.$  
  \item[b)] For all other $k\neq 1,$ $v(u_{k})=0;$ and  $v(u^{*}_{1})=t^{-1}e_{1}.$ 
\end{enumerate}  
   So this open orbit corresponds to the element $w^{'}=(\lambda^{'},\tau\circ \nu)$ in $X_{G}\times S_{n,m},$ where $\lambda^{'}=(-1,0,\dots,0).$ This leads to the desired distinguished triangle. 
 \item 
  In this case  $\heckefunc{H}(L_{w},\mathcal{I}^{\mu!})$ is naturally a complex over $Hom_{\locring}(L^{*},U_{m-1}/U_{-1}).$ The space $Hom_{\locring}(L^{*},U_{m-1}/U_{-1})$ is an $\locring$-module on which $t$ acts trivially, hence is a vector space. Denote by $Y'$ the subscheme of $Hom_{\locring}(L^*,U_{m-1}/U_{-1})$ given by the conditions:
 \begin{enumerate}
 \item[a)] $v(e_{n}^{*})\in{U_{1}-U_{-1}}.$ 
 \item [b)] $v(e_{j}^{*})\in{U_{\nu( j)}-U_{\nu (j)-1}}$ for $1\leq j < n$. 
 \end{enumerate}
Then the fibres of the map $\eqref{map1m}$  identify with $\mathbb{A}^{1}.$ So $\overset{\leftarrow}{H}_{H}(L_{w},\mathcal{I}^{\mu !})$ is the sheaf $\IC(Y^{'})[-1]$ extended by zero to $Hom_{\locring}(L^{*},U_{m-1}/U_{-1}).$ The scheme $Y^{'}$ is the union of two $I_{H}\times I_{G}$-orbits, the open corresponding to $\nu$ and the closed one passing though $v$ given by 
  $$v(e_{n}^{*})=tu_{m},v(e_{n-1}^{*})=u_{\nu(n-1)},\dots ,v(e_{1}^{*})=u_{\nu (1)}. $$
  Let $w^{'}=(\lambda^{'},\tau\circ \nu)$ be in  $X_{G}\times S_{n,m}$ with $\lambda^{'}$ being equal to $(0,\dots,0,1).$ Then we have a distinguished triangle
$$\mathcal{I}^{\mu}[-1]\longrightarrow{\overset{\leftarrow}{H}_{H}(L_{w},\mathcal{I}^{\mu !})}\longrightarrow {\mathcal{I}^{w^{'}}}\overset{+1}{\longrightarrow}.$$
\item   
We have $\nu (1)=m$ and $\nu(n)=1.$ Let $Y^{'}$ be the subscheme of $Hom_{\locring}(L^{*},U_{m+1}/U_{-1})$ classifying maps $v$ satisfying the conditions:

\begin{enumerate}
\item [a)] For all $1< j< n,$  $v(e^{*}_{j})\in{U_{\nu( j)}}-U_{\nu (j)-1}.$  
\item [b)] $v(e_{n}^{*})\in{U_{1}-U_{-1}}.$ 
\item[c)] $v(e_{1}^{*})\in{U_{m+1}-U_{ m-1}}.$ 
\item[d)] $\{v(e_{n}^{*}),tv(e_{1}^{*})\}$ are linearly independent in $U_{1}/U_{-1}.$
\end{enumerate}

Then the map $\eqref{map1m}$ is an isomorphism onto $Y'.$ Hence $\overset{\leftarrow}{H}_{H}(L_{w},\mathcal{I}^{\mu !})$ identifies with $\IC(Y')$ extended by zero from $Y'$ to $Hom_{\locring}(L^{*},U_{m+1}/U_{-1}).$
The scheme $Y'$ contains the closed subscheme which  is the $I_{H}\times I_{G}$-orbit  corresponding to $\nu$. The complement of the latter scheme in $Y^{'}$ is the $I_{H}\times I_{G}$-orbit passing through $v$ given by 
$$v(e_{1}^{*})=t^{-1}u_{1}, v(e_{2})=u_{\nu(2)},\dots, v(e_{n-1}^{*})=u_{\nu(n-1)}, v(e_{n}^{*})=tu_m.$$
Let $w^{'}=(\lambda,\tau_{i}\circ \nu)$ where $\lambda=(-1,0,\dots,0,1 ).$ Then there is a distinguished triangle $$\mathcal{I}^{w^{'} !}\longrightarrow \IC(Y^{'})\longrightarrow \mathcal{I}^{\mu !}[1]\overset{+1}{\longrightarrow}.$$
\end{enumerate}
\end{proof}

%\bibliographystyle{alpha} 
%\bibliography{biblio}

\end{document}